\documentclass[12pt]{article}
\usepackage[utf8]{inputenc}
\usepackage{amsmath}
\usepackage{amssymb}
\usepackage{amsthm}
\usepackage{commath}
\usepackage{enumitem}
\usepackage[maxnames=5]{biblatex}
\usepackage{bm}
\usepackage[margin=3cm]{geometry}
\usepackage{array} 
\usepackage{blkarray}
\usepackage{setspace}
\newcolumntype{L}{>{$}l<{$}}
\newtheorem{theorem}{Theorem}[section]
\newtheorem{definition}[theorem]{Definition}
\newtheorem{proposition}[theorem]{Proposition}
\newtheorem{corollary}[theorem]{Corollary}

\newtheorem{lemma}[theorem]{Lemma}
\newtheorem{remark}[theorem]{Remark}
\newcommand{\Irr}{\mathrm{Irr}}
\newcommand{\IBr}{\mathrm{IBr}}
\newcommand{\Alt}{\mathrm{Alt}}
\newcommand{\Sym}{\mathrm{Sym}}
\newcommand{\PSL}{\mathrm{PSL}}
\newcommand{\PSU}{\mathrm{PSU}}
\newcommand{\SL}{\mathrm{SL}}
\newcommand{\SU}{\mathrm{SU}}
\newcommand{\GL}{\mathrm{GL}}
\newcommand{\PGL}{\mathrm{PGL}}
\newcommand{\GU}{\mathrm{GU}}
\newcommand{\Out}{\mathrm{Out}}
\newcommand{\Aut}{\mathrm{Aut}}
\usepackage{hyperref}

\title{Morita Equivalence Classes of Tame Blocks of Finite Groups}
\author{Norman Macgregor}
\date{\today}

\begin{document}

\maketitle

\begin{abstract}
\noindent We show that several Morita equivalence classes of tame algebras do not occur as blocks of finite groups. This refines classifications by Erdmann of classes of blocks with dihedral, semidihedral, and generalised quaternion defect groups. In particular we now have a complete classification of the Morita equivalence classes of blocks of finite groups with dihedral defect groups.
\end{abstract}

\tableofcontents

\section{Introduction}

Throughout we work over an algebraically closed field $k$ of characteristic $\ell>0$. A finite-dimensional algebra over $k$ has representation type one of finite, tame, or wild. It is called tame if it has infinitely many indecomposable modules but, for each $n\in\mathbb{N}$, all but finitely many of the indecomposable modules of dimension $n$ fall into finitely many one-parameter families. A block of a finite group has tame representation type if and only if its defect groups are dihedral, semidihedral, or generalised quaternion \cite{Bondarenko-Drozd}.

In a series of papers -- the details of which are collected in \cite{Erdmannbook} -- Erdmann classified several families of Morita equivalence classes of tame algebras which include all tame blocks of finite groups, though it is not known precisely which classes actually occur as blocks. We improve each of the classifications of blocks with dihedral, semidihedral, and generalised quaternion defect groups by using the Classification of Finite Simple Groups to show that several classes do not occur as blocks of any quasi-simple groups, and hence, via a reduction to quasi-simple groups, do not occur as blocks of any finite groups. 

\begin{theorem}\label{maintheorem}
There are no blocks of finite groups with the following defect groups of order $2^n$ in the following Morita equivalence classes (according to Erdmann's labellings \cite{Erdmannbook}):
\begin{itemize}
\item Dihedral: $(3B)$ for $n\geq4$;
\item Generalised quaternion: $(3B)$ for $n\geq5$;
\item Semidihedral: $(2B_4)$ and $(3H)$ for any $n$;
$(2B_1)$, $(3B_2)$, and $(3C_{2,1})$ for $n\geq5$.
\end{itemize}
In particular, the Morita equivalence classes that occur as blocks of finite groups with dihedral defect groups are known.
\end{theorem}

Donovan's conjecture states that for any $\ell$-group $D$ there are, up to Morita equivalence, only finitely many blocks of finite groups with $D$ as a defect group; this is known to be true for many families of groups. Erdmann's classifications already proved this for dihedral and semidihedral groups, however the Morita equivalence classes were not fully and precisely classified, as we now have for all dihedral groups. This is perhaps the first `interesting' family of non-abelian groups for which this has been done; interesting in that there are multiple classes and they include blocks of finite simple groups.

We give updated classifications of the remaining classes of blocks that may occur in Section \ref{classifications}, together with their decomposition matrices and those of the eliminated classes. Then, after recalling various background results in Section \ref{preliminaries}, we prove Theorem \ref{maintheorem} by first proving reductions to quasi-simple groups in Section \ref{reductions} -- in fact it is sufficient to consider only the simple groups and their odd covers -- then in Section \ref{proof} we work through the blocks of all these groups case by case, describing the blocks that (might) have dihedral or semidihedral defect groups and showing that the Morita equivalence classes in question do not appear, and thus do not occur as blocks of any finite group. This is done primarily by deducing the decomposition matrices from the ordinary character degrees, using descriptions of blocks from various papers. The generalised quaternion case is obtained as a corollary of the dihedral case. For the semidihedral class $(2B_1)$, however, quasi-simple groups are not quite sufficient, and this case is considered finally in Section \ref{2B1}.

\section{The Classifications}\label{classifications}

We refine Erdmann's classifications of blocks with dihedral, semidihedral, and generalised quaternion defect groups, each of which can be found together in her book \cite{Erdmannbook}; additionally see \cite[Section 6.2]{Craven} for a succinct description of the previously most up-to-date versions and example blocks. We obtain the following now complete classification of Morita equivalence classes of blocks with dihedral defect groups:

\begin{theorem}\label{classification}
If $B$ is a block of a finite group with dihedral defect groups of order $2^n$ for $n\geq3$, then exactly one of the following holds: 
\begin{enumerate}
\item[$(1A)$] $B$ is Morita equivalent to $kD_{2^n}$;
\item[$(2A)$] $B$ is Morita equivalent to the principal block of $\PGL_2(q)$, for $q\equiv1\bmod4$ where $\abs{q-1}_2=2^{n-1}$;
\item[$(2B)$] $B$ is Morita equivalent to the principal block of $\PGL_2(q)$, for $q\equiv-1\bmod4$ where $\abs{q+1}_2=2^{n-1}$;
\item[$(3A)$] $B$ is Morita equivalent to the principal block of $\PSL_2(q)$, for $q\equiv1\bmod4$ where $\abs{q-1}_2=2^n$;
\item[$(3K)$] $B$ is Morita equivalent to the principal block of $\PSL_2(q)$, for $q\equiv-1\bmod4$ where $\abs{q+1}_2=2^n$;
\item[$(3B)$] $n=3$ and $B$ is Morita equivalent to the principal block of $\Alt(7)$.
\end{enumerate}
The decomposition matrix of $B$ is then accordingly one of the following, in order; in each case the last row is repeated $2^{n-2}-1$ times: 
\begin{equation*}
\begin{pmatrix}
1\\1\\1\\1\\2
\end{pmatrix},
\qquad
\begin{pmatrix}
1&.\\1&.\\1&1\\1&1\\2&1
\end{pmatrix},
\qquad
\begin{pmatrix}
1&.\\1&.\\1&1\\1&1\\.&1
\end{pmatrix},
\qquad
\begin{pmatrix}
1&.&.\\1&1&.\\1&.&1\\1&1&1\\2&1&1
\end{pmatrix},
\qquad
\begin{pmatrix}
1&.&.\\.&1&.\\.&.&1\\1&1&1\\.&1&1
\end{pmatrix},
\qquad
\begin{pmatrix}
1&.&.\\1&1&.\\1&.&1\\1&1&1\\.&1&.
\end{pmatrix}.
\end{equation*}
\end{theorem}

Note the `.' denote zero entries, and $\abs{m}_p$ denotes the largest power of a prime $p$ dividing an integer $m$. The class we eliminate is $(3B)$ for $n\geq4$. In Erdmann's original classification there was also a parameter resulting in two possible classes for each decomposition matrix with two simple modules, but these additional classes were eliminated by Eisele in \cite{Eisele}. Now there is an example of a block occurring in each possible class and each possible $n$, so we have a complete classification.

Throughout when we say dihedral we usually refer to groups of order at least $8$, though blocks with Klein four defect groups are also tame. In that case only classes $(1A)$, $(3A)$, and $(3K)$ in Theorem \ref{classification} occur, as the principal blocks of $V_4$, $\Alt(5)\cong \PSL_2(5)$, and $\Alt(4)\cong \PSL_2(3)$ respectively \cite{Erdmannbook}; note that since $n=2$ the decomposition matrices have only four rows.

Theorem \ref{classification} implies a similar improvement on Erdmann's classification of blocks with generalised quaternion defect groups, completely classifying those with three simple modules by eliminating the corresponding class (3B) for $n\geq5$:

\begin{corollary}\label{quaternion}
If $B$ is a block of a finite group with generalised quaternion defect groups of order $2^n$ for $n\geq3$ and $l(B)=3$, then exactly one of the following holds: 

\begin{enumerate}
\item[$(3A)$] $B$ is Morita equivalent to the principal block of $\SL_2(q)$, for $q\equiv1\bmod4$ where $\abs{q-1}_2=2^{n-1}$;
\item[$(3K)$] $B$ is Morita equivalent to the principal block of $\SL_2(q)$, for $q\equiv-1\bmod4$ where $\abs{q+1}_2=2^{n-1}$;
\item[$(3B)$] $n=4$ and $B$ is Morita equivalent to the principal block of $2\cdot\Alt(7)$.
\end{enumerate}
The decomposition matrix of $B$ is then accordingly one of the following; in each case the last row is repeated $2^{n-2}-1$ times:
\begin{equation*}
\begin{pmatrix}
1&.&.\\1&1&.\\1&.&1\\1&1&1\\.&1&.\\.&.&1\\2&1&1
\end{pmatrix},
\qquad
\begin{pmatrix}
1&.&.\\.&1&.\\.&.&1\\1&1&1\\1&1&.\\1&.&1\\.&1&1
\end{pmatrix},
\qquad
\begin{pmatrix}
1&.&.\\1&1&.\\1&.&1\\1&1&1\\2&1&1\\.&.&1\\.&1&.
\end{pmatrix}.
\end{equation*}
\end{corollary}

This follows from a result of Kessar and Linckelmann \cite{Kessar-Linckelmann}: a block $B$ of a finite group $G$ with generalised quaternion defect group $D$ and $l(B)=3$ must be Morita equivalent to its Brauer correspondent in $C_G(Z(D))$. This implies that $C_G(Z(D))/Z(D)$ has a block $b$ contained in $B$ with defect group $D/Z(D)$, which is dihedral. Thus if $B$ were in class (3B) with defect $n\geq5$, then $b$ would have to be in dihedral class (3B) with defect $n-1\geq4$, which is impossible by Theorem \ref{classification}.

The decomposition matrices for blocks with generalised quaternion defect groups and $l(B)\leq2$ are as follows:

\begin{equation*}
\begin{pmatrix}
1\\1\\1\\1\\2
\end{pmatrix},
\qquad
\begin{pmatrix}
1&.\\1&.\\1&1\\1&1\\.&1\\2&1
\end{pmatrix},
\qquad
\begin{pmatrix}
1&.\\1&.\\1&1\\1&1\\2&1\\.&1
\end{pmatrix}.
\end{equation*}

These are labelled $(1A)$, $(2A)$, and $(2B)$, and occur as $kQ_{2^n}$ and, for $n\geq4$, as the principal blocks of $\SL_2(q).2$ for $q\equiv1$ or $-1\bmod4$ respectively; note that these were mistakenly listed as $\SL_2(q^2).2$ in the original version of \cite{Craven}. While we have an example of a block with each possible decomposition matrix for generalised quaternion defect groups, for each matrix with $l(B)=2$ there is a parameter giving infinitely many possible Morita equivalence classes, so we do not have a complete classification and even Donovan's conjecture is unknown.

We rule out several possible Morita equivalence classes of blocks with semidihedral defect groups, though in this case we have a less complete result.

\begin{theorem}\label{semidihedral}
If $B$ is a block of a finite group with semidihedral defect groups of order $2^n$ for $n\geq4$ and $l(B)=3$, then exactly one of the following holds: 
\begin{enumerate}
\item[$(3B_1)$] $B$ is Morita equivalent to the principal block of $\PSL_3(q)$, for $q\equiv-1\bmod4$ where $\abs{q+1}_2=2^{n-2}$;
\item[$(3D)$] $B$ has the same decomposition matrix as in $(3B_1)$, but with a different Ext-quiver; no example of such a block is known;
\item[$(3A_1)$] $B$ is Morita equivalent to the principal block of $\PSU_3(q)$, for $q\equiv1\bmod4$ where $\abs{q-1}_2=2^{n-2}$;
\item[$(3C_{2,2})$] no example of a block in the Morita equivalence class of $B$ is known;
\item [$(*)$] $n=4$ and $B$ is Morita equivalent to a certain non-principal block of the Monster group.
\end{enumerate}
The block $(*)$ of the Monster is either in class $(3B_2)$ or $(3C_{2,1})$, and thus for $n=4$ there are blocks in one of these classes but not both, and there are no blocks in either class for $n\geq5$. The decomposition matrix of $B$ is then one of the following, with the last row repeated $2^{n-2}-1$ times; from left to right: $(3B_1)$ or $(3D)$, $(3A_1)$, $(3C_{2,2})$, $(3B_2)$, $(3C_{2,1})$: 
\begin{equation*}
\begin{pmatrix}
1&.&.\\1&1&.\\1&.&1\\1&1&1\\.&.&1\\.&1&.
\end{pmatrix},
\quad
\begin{pmatrix}
1&.&.\\1&1&.\\1&.&1\\1&1&1\\.&.&1\\2&1&1
\end{pmatrix},
\quad
\begin{pmatrix}
1&.&.\\.&1&.\\1&.&1\\.&1&1\\.&.&1\\1&1&1
\end{pmatrix},
\quad
\begin{pmatrix}
1&.&.\\1&1&.\\1&.&1\\1&1&1\\2&1&1\\.&1&.
\end{pmatrix},
\quad
\begin{pmatrix}
1&.&.\\.&1&.\\1&.&1\\.&1&1\\1&1&1\\.&.&1
\end{pmatrix}.
\end{equation*}
\end{theorem}

Note that \cite{Erdmannbook} incorrectly lists the principal block of $M_{11}$ in class $(3D)$, but it is in $(3B_1)$; this can be checked in Magma. The classes eliminated are $(3B_2)$ and $(3C_{2,1})$ for $n\geq5$, and $(3H)$ for all $n$ which has the following decomposition matrix:
\begin{equation*}
\begin{pmatrix}
1&.&.\\.&1&.\\.&.&1\\1&1&1\\.&1&1\\1&1&.
\end{pmatrix}.
\end{equation*}

Here, unlike the dihedral case, the Morita equivalence class cannot always be identified based only on the degrees of the ordinary characters in the block. The same list of ordinary degrees could give rise to the decomposition matrices of both $(3A_1)$ and $(3C_{2,2})$, and similarly for $(3B_2)$ and $(3C_{2,1})$, while $(3B_1)$ and $(3D)$ even have identical decomposition matrices.

\begin{theorem}\label{semidihedral2}
If $B$ is a block of a finite group with semidihedral defect groups of order $2^n$ for $n\geq4$ and $l(B)\leq2$, then exactly one of the following holds: 
\begin{enumerate}
\item[$(1A)$] $B$ is Morita equivalent to $kSD_{2^n}$;
\item[$(2A_1)$] $B$ is Morita equivalent to the principal block of $\GU_2(q)$, for $q\equiv1\bmod4$ where $\abs{q-1}_2=2^{n-2}$;
\item[$(2B_2)$] $B$ is Morita equivalent to the principal block of $\GL_2(q)$, for $q\equiv-1\bmod4$ where $\abs{q+1}_2=2^{n-2}$;
\item[$(2A_2)$] $B$ is Morita equivalent to the principal block of $\PSL_2(q^2).2$, for $q$ odd where $\abs{q^2-1}_2=2^{n-1}$;
\item[$(2B_1)$] $n=4$ and $B$ is Morita equivalent to a non-principal block of $3\cdot M_{10}$;
\item[$(*)$] $B$ is Morita equivalent to an algebra with the same decomposition matrix and Ext-quiver as one of the four previous two-module classes, but with a different parameter in the relations for the path algebra; no example of such a block is known.
\end{enumerate}
The decomposition matrix of $B$ is then accordingly one of the following, with the last row repeated $2^{n-2}-1$ times: 
\begin{equation*}
\begin{pmatrix}
1\\1\\1\\1\\2
\end{pmatrix},
\qquad
\begin{pmatrix}
1&.\\1&.\\1&1\\1&1\\.&1\\2&1
\end{pmatrix},
\qquad
\begin{pmatrix}
1&.\\1&.\\1&1\\1&1\\2&1\\.&1
\end{pmatrix},
\qquad
\begin{pmatrix}
1&.\\1&.\\1&1\\1&1\\2&1
\end{pmatrix},
\qquad
\begin{pmatrix}
1&.\\1&.\\1&1\\1&1\\.&1
\end{pmatrix}.
\end{equation*}
\end{theorem}

For $l(B)\leq2$ we have an example of a block for each possible decomposition matrix, but each with two simple modules has two possible Morita equivalence classes. The classes eliminated are (both families of) $(2B_1)$ for $n\geq5$, and $(2B_4)$ for all $n$ which has the following decomposition matrix:
\begin{equation*}
\qquad
\begin{pmatrix}
1&.\\1&.\\.&1\\.&1\\1&1
\end{pmatrix}.
\end{equation*}

\section{Preliminaries}\label{preliminaries}
\subsection{Algebraic Groups}

Throughout this section let $\mathbf{G}$ be a connected reductive algebraic group in characteristic $p\neq\ell$, with a Steinberg endomorphism $F:\mathbf{G}\rightarrow\mathbf{G}$ so that $\mathbf{G}^F$ is a finite group of Lie type, and let $\mathbf{G}^*$ be a group dual to $\mathbf{G}$ with Steinberg endomorphism also denoted by $F$.

For $\mathbf{L}$ an $F$-stable Levi subgroup of $\mathbf{G}$, Deligne-Lusztig induction is a functor from virtual characters of $\mathbf{L}^F$ to those of $\mathbf{G}^F$ generalising Harish-Chandra induction (details can be found in \cite{Digne-Michel} for example). For ease of notation we will denote this functor by $R_\mathbf{L}^\mathbf{G}$, instead of $R_{\mathbf{L}^F}^{\mathbf{G}^F}$. If $(\mathbf{L},\lambda)$ is a cuspidal pair (as in Theorem \ref{Kessar-Malle}) then the set of constituents of $R_\mathbf{L}^\mathbf{G}(\lambda)$ will be called a Harish-Chandra series.

Let $\mathbf{T},\mathbf{S}$ be $F$-stable maximal tori of $\mathbf{G}$. If $\theta\in\Irr(\mathbf{T}^F)$ and $\psi\in\Irr(\mathbf{S}^F)$ correspond to $t\in(\mathbf{T}^*)^{F}$ and $s\in(\mathbf{S}^*)^{F}$, via suitable isomorphisms \cite[(8.14)]{Cabanes-Enguehard}, then the pairs $(\mathbf{T},\theta)$ and $(\mathbf{S},\psi)$ are said to be rationally conjugate whenever $t$ and $s$ are $(\mathbf{G}^*)^{F}$-conjugate; thus rational conjugacy classes of pairs $(\mathbf{T},\theta)$ correspond to conjugacy classes of semisimple elements in $(\mathbf{G}^*)^F$.

\begin{definition}
Let $s\in(\mathbf{G}^*)^F$ be semisimple. The rational Lusztig series $\mathcal{E}(\mathbf{G}^F,s)$ associated to the $(\mathbf{G}^*)^F$-conjugacy class of $s$ is the set of irreducible characters of $\mathbf{G}^F$ that occur as a constituent in some $R_\mathbf{T}^\mathbf{G}(\theta)$, where $(\mathbf{T},\theta)$ is in the rational conjugacy class associated to (the conjugacy class of) $s$.

If $s$ is a semisimple $\ell^\prime$-element of $(\mathbf{G}^*)^F$ define 
\begin{equation*}
\mathcal{E}_\ell(\mathbf{G}^F,s)=\bigcup_{t\in(C_{\mathbf{G}^*}(s)^F)_\ell}\mathcal{E}(\mathbf{G}^F,st),
\end{equation*}
so $t$ runs over all $\ell$-elements of $(\mathbf{G}^*)^F$ that commute with $s$.
\end{definition}

These series describe the blocks of $\mathbf{G}^F$ in the following way, where here a block $B$ is thought of as the set $\Irr(B)$: 

\begin{theorem}[{\cite[Thms 8.24 \& 9.12]{Cabanes-Enguehard}}]
The sets $\mathcal{E}(\mathbf{G}^F,s)$ for semisimple $s\in(\mathbf{G}^*)^F$ form a partition of $\Irr(\mathbf{G}^F)$. If $s\in(\mathbf{G}^*)^F$ is a semisimple $\ell^\prime$-element then:
\begin{enumerate}[label=(\roman*)]
\item $\mathcal{E}_\ell(\mathbf{G}^F,s)$ is a union of $\ell$-blocks of $\mathbf{G}^F$;
\item each $\ell$-block contained in $\mathcal{E}_\ell(\mathbf{G}^F,s)$ contains an element of $\mathcal{E}(\mathbf{G}^F,s)$.
\end{enumerate}
\end{theorem}

Hence to parameterise the $\ell$-blocks of $\mathbf{G}^F$ it suffices to decompose $\mathcal{E}(\mathbf{G}^F,s)$ into $\ell$-blocks for each semisimple $\ell^\prime$-element $s\in(\mathbf{G}^*)^F$.

\begin{definition}
An irreducible character of $\mathbf{G}^F$ is called unipotent if it is a constituent of $R_\mathbf{T}^\mathbf{G}(1)$ for some $F$-stable maximal torus $\mathbf{T}$ of $\mathbf{G}$. A block is called unipotent if it contains a unipotent character.
\end{definition}

Hence $\mathcal{E}(\mathbf{G}^F,1)$ is the set of unipotent characters of $\mathbf{G}^F$. The following so-called Jordan decomposition by Lusztig relates characters to unipotent characters of a usually smaller group:

\begin{theorem}[{\cite[(4.23)]{Lusztig}}]\label{JordanDecomposition}
If $\mathbf{G}$ has connected centre and $s\in(\mathbf{G}^*)^F$ is semisimple, then there is a bijection between $\mathcal{E}(\mathbf{G}^F,s)$ and $\mathcal{E}(C_{\mathbf{G}^*}(s)^F,1)$ such that if $\chi$ is mapped to $\chi_u$ then
\begin{equation*}
\chi(1)=\chi_u(1)\cdot\left|(\mathbf{G}^*)^F:C_{\mathbf{G}^*}(s)^F\right|_{p^\prime}.
\end{equation*}
\end{theorem}

The requirement that $\mathbf{G}$ has connected centre guarantees that $C_{\mathbf{G}^*}(s)$ is connected \cite[Thm 4.5.9]{Carter}. To allow otherwise we must define unipotent characters for a disconnected group $\mathbf{G}$: they are the constituents of $(R_\mathbf{T}^{\mathbf{G}^\circ}(1)){\uparrow}^{\mathbf{G}^F}$ for $F$-stable maximal tori $\mathbf{T}$ of $\mathbf{G}^\circ$.

Since any $s\in(\mathbf{G}^*)^F_{l^\prime}$ and $t\in C_{\mathbf{G}^*}(s)^F_\ell$ have coprime orders and commute, if $C_{\mathbf{G}^*}(s)^F$ is connected then there are correspondences between $\mathcal{E}(\mathbf{G}^F,st)$ and $\mathcal{E}(C_{\mathbf{G}^*}(st)^F,1)=\mathcal{E}(C_{C_{\mathbf{G}^*}(s)}(t)^F,1)$ and $\mathcal{E}((C_{\mathbf{G}^*}(s)^*)^F,t)$. Hence there is also a bijection between $\mathcal{E}_\ell(\mathbf{G}^F,s)$ and $\mathcal{E}_\ell((C_{\mathbf{G}^*}(s)^*)^F,1)$. 

Bonnaf\'{e} and Rouquier gave an important reduction showing that if the centraliser is contained in a proper $F$-stable Levi subgroup then there is a Morita equivalence between the related blocks:

\begin{theorem}[\cite{Bonnafe-Rouquier}]\label{Bonnafe-Rouquier}
Let $s\in(\mathbf{G}^*)^F$ be a semisimple $\ell^\prime$-element and suppose that $C_{\mathbf{G}^*}(s)$ is contained in an $F$-stable Levi subgroup $\mathbf{L}^*$ of $\mathbf{G}^*$. The $\ell$-blocks of $\mathbf{G}^F$ in $\mathcal{E}(\mathbf{G}^F,s)$ are in bijection via Jordan correspondence with the $\ell$-blocks of $\mathbf{L}^F$ in $\mathcal{E}(\mathbf{L}^F,s)$, and the corresponding blocks are Morita equivalent.
\end{theorem}

This Morita equivalence preserves decomposition matrices and, as shown further with Dat in \cite{Bonnafe-Dat-Rouquier}, also defect groups. The result motivates the following definition:

\begin{definition}
A semisimple element $s\in\mathbf{G}$ is quasi-isolated if its centraliser $C_\mathbf{G}(s)$ is not contained in any proper Levi subgroup of $\mathbf{G}$; further, it is isolated if the identity component $C_\mathbf{G}(s)^\circ$ is not contained in a proper Levi subgroup. A block is then (quasi)-isolated if its semisimple label is (quasi)-isolated.
\end{definition}

Unipotent blocks are isolated, and if $\mathbf{G}$ has connected centre then $C_{\mathbf{G}^*}(s)$ is always connected so the terms isolated and quasi-isolated coincide. Note that if $s\in\mathbf{G}^F$ and $C_\mathbf{G}(s)$ is contained in a Levi subgroup $\mathbf{L}$ of $\mathbf{G}$, then $C_\mathbf{G}(s)$ is indeed also contained in an \textit{F-stable} Levi subgroup $\mathbf{L}\cap F(\mathbf{L})$. 

Additionally, the above equivalence is preserved taking quotients by central $\ell$-subgroups:

\begin{theorem}[{\cite[Prop. 4.1]{Eaton-Kessar-Kulshammer-Sambale}}]\label{BonnafeRouquiercentral}
If $B$ and $C$ are $\ell$-blocks of $\mathbf{G}^F$ and $\mathbf{L}^F$ in correspondence by Theorem \ref{Bonnafe-Rouquier}, and $Z$ is a central $\ell$-subgroup of $\mathbf{G}^F$, then the $\ell$-blocks $\overline{B}$ and $\overline{C}$ of $\mathbf{G}^F/Z$ and $\mathbf{L}^F/Z$ contained in $B$ and $C$ respectively are Morita equivalent. 
\end{theorem}

In light of the above, if a block $B$ of $\mathbf{G}^F$ is (quasi)-isolated, then we will also say that the block $\overline{B}$ of $\mathbf{G}^F/Z$ contained in $B$ is (quasi)-isolated. Quasi-isolated blocks of finite groups of Lie type have been well studied; in particular, Kessar and Malle describe the quasi-isolated $\ell$-blocks of exceptional groups of Lie type when $\ell$ is a bad prime for $\mathbf{G}$, obtaining the following description:

\begin{theorem}[{\cite[Thm 1.2]{Kessar-Malle}}]\label{Kessar-Malle}
Let $\mathbf{G}$ be a simply connected simple exceptional algebraic group, $\ell\neq p$ be a bad prime for $\mathbf{G}$, and $1\neq s\in(\mathbf{G}^*)^F$ be a quasi-isolated $\ell^\prime$-element.
\begin{enumerate}[label=(\roman*)]
\item There is a bijection between $\ell$-blocks of $\mathbf{G}^F$ in $\mathcal{E}_\ell(\mathbf{G}^F,s)$ and pairs $(\mathbf{L},\lambda)$ where $\mathbf{L}$ is an $e$-split Levi subgroup of $\mathbf{G}$ and $\lambda\in\mathcal{E}_\ell(\mathbf{L}^F,s)$ is $e$-cuspidal of quasi-central $\ell$-defect.
\item If a block corresponds to $(\mathbf{L},\lambda)$ then it has a defect group $D$ such that $Z(\mathbf{L})_\ell^F=Z\trianglelefteq P\trianglelefteq D$, where $D/P$ is isomorphic to a Sylow $\ell$-subgroup of the Weyl group $W_{\mathbf{G}^F}(\mathbf{L},\lambda)$ and $P/Z$ is isomorphic to a Sylow $\ell$-subgroup of $\mathbf{L}^F/Z[\mathbf{L},\mathbf{L}]^F$.
\end{enumerate}
\end{theorem}

See \cite{Kessar-Malle} for definitions of the terminology used in (i); we will not use them further. For each exceptional group they list all $(\mathbf{L},\lambda)$ that occur, and give further information.

Ennola duality \cite{Broue-Malle-Michel}, formally swapping $q$ with $-q$, is used in \cite{Kessar-Malle} and will be used later to translate arguments to different values of $q\bmod4$ (which correspond to different values of $e$ in the above theorem); all groups in question are exchanged with their Ennola duals, so for example $E_6(q)$ becomes $^2\!E_6(q)$ and $q-1$ becomes $q+1$, while $G_2(q)$ is its own Ennola dual.

\subsection{Normal Subgroups}

Let $N$ be a normal subgroup of a finite group $G$ and let $\ell$ be a prime. We collect some facts relating the characters and blocks of $N$ and $G$ that will be used throughout.

Of particular note, if $\abs{G:N}=2$ then, from Clifford's Theorem \cite[Thm 20.8]{James-Liebeck}, for any $\varphi\in\Irr(N)$: either $\varphi{\uparrow}^G=\chi_1+\chi_2$ and $\chi_1{\downarrow}_N=\chi_2{\downarrow}_N=\varphi$ for some distinct $\chi_1,\chi_2\in\Irr(G)$ of equal degree; or $\varphi{\uparrow}^G=\psi{\uparrow}^G=\chi$ and $\chi{\downarrow}_N=\varphi+\psi$ for some $\chi\in\Irr(G)$ and $\varphi\neq\psi\in\Irr(N)$ with $\varphi,\psi$ of equal degree. We will say these characters either split or fuse on induction or restriction accordingly. 

\begin{lemma}[{\cite[Thm 2.4.7]{Linckelmann1}}]\label{inductionsimple}
Let $\phi\in\Irr(N)$. Then $\phi{\uparrow}^G\in\Irr(G)$ if and only if $\phi^g\neq\phi$ for all $g\in G\setminus N$.
\end{lemma}

\begin{lemma}[{\cite[Thm 9.4]{Navarro}}]\label{coveringblock}
Let $B$ be a block of $G$ covering a block $b$ of $N$ and $\varphi\in\Irr(b)\cup\IBr(b)$. Then $\varphi$ is a constituent of $\chi{\downarrow}_N$ for some $\chi\in\Irr(B)\cup\IBr(B)$.
\end{lemma}

\begin{lemma}[{\cite[Thm 8.11, Cor. 9.6, Thm 9.17]{Navarro}}]\label{evenindex}
Let $\abs{G:N}$ be a power of $\ell$.
\begin{enumerate}[label=(\roman*)]
\item If $\varphi\in\IBr(N)$ then there is a unique $\chi\in\IBr(G)$ covering $\varphi$ and $\chi{\downarrow}_N$ is the sum of the distinct $G$-conjugates of $\varphi$.
\item If $b$ is an $\ell$-block of $N$ then there is a unique $\ell$-block $B$ of $G$ covering $b$.
\item Furthermore, if $b$ is $G$-invariant with defect group $D$, then the defect groups of $B$ have order $\abs{D}\cdot\abs{G:N}$.
\end{enumerate}
\end{lemma}

Now let $\overline{G}=G/N$. The characters of $\overline{G}$ inflate to characters of $G$, so $\Irr(\overline{G})$ and $\IBr(\overline{G})$ are viewed as subsets of $\Irr(G)$ and $\IBr(G)$, and similarly for their blocks; each block of $\overline{G}$ is contained in a unique block of $G$. 

\begin{lemma}[{\cite[Thm 9.9]{Navarro}}]
\begin{enumerate}[label=(\roman*)]
\item If $N$ is an $\ell$-group, then each $\ell$-block $B$ of $G$ contains an $\ell$-block of $\overline{G}$ whose defect groups are $D/N$ for defect groups $D$ of $B$. 
\item If $N$ is an $\ell^\prime$-group and $\overline{B}$ is an $\ell$-block of $\overline{G}$ contained in an $\ell$-block $B$ of $G$, then $\Irr(\overline{B})=\Irr(B)$ and $\IBr(\overline{B})=\IBr(B)$.
\end{enumerate}
\end{lemma}

\begin{lemma}[\cite{Knorr}]\label{Knorr}
Let $B$ be a block of $G$ covering a block $b$ of $N$. If $P$ is a defect group of $b$ then there is a defect group $D$ of $B$ such that $P=D\cap N$. Conversely, if $D$ is a defect group of $B$ then $D\cap N$ is a defect group of a block $b^g$ for some $g\in G$.
\end{lemma}

\begin{lemma}[Fong's First Reduction {\cite[Thm 7.4.2]{Craven}}]\label{Fong1}
Let $b$ be a block of $N$ and define the inertia subgroup of $b$ in $G$ as $T=\{g\in G\mid b^g=b\}$. There is a correspondence between blocks of $G$ covering $b$ and blocks of $T$ covering $b$, with corresponding blocks being Morita equivalent and having isomorphic defect groups.
\end{lemma}
Note that $N\leq T$ and, since covered blocks are conjugate, $b$ is the unique block of $N$ covered by these blocks of $T$.

\begin{lemma}[Fong's Second Reduction {\cite[Thm 7.4.4]{Craven}}]\label{Fong2}
Let $B$ be a block of $G$ with a defect group $D$, such that $D\cap N=1$. Then there is a block $\hat{B}$ of some central extension of $G/N$ such that $B$ and $\hat{B}$ are Morita equivalent.
\end{lemma}

\subsection{Tame Blocks}

As described in \cite{Olsson}, for any block $B$ with dihedral, semidihedral, or generalised quaternion defect groups of order $2^n$ the numbers of irreducible ordinary and Brauer characters are $2^{n-2}+3\leq k(B)\leq2^{n-2}+5$ and $1\leq l(B)\leq3$ respectively. There are four ordinary characters of height zero (the first four rows in each of the decomposition matrices in Section \ref{classifications}) and $2^{n-2}-1$ height one characters (the repeated row), and any additional characters have height $n-2$; blocks with dihedral defect groups have none of these `large height' characters, semidihedral blocks have at most one, and generalised quaternion blocks have at most two. Semidihedral defect groups are of order at least $2^4$, so the `large height' is in particular greater than one. 

If $B$ has dihedral defect groups then its decomposition matrix, hence Morita equivalence class, can be determined -- out of those in Theorem \ref{classification} and including $(3B)$ as a possibility for any $n$ -- using only its ordinary character degrees, such as by checking the following properties: if the degree of the height one character is the largest, then $B$ is in one of $(1A)$, $(2A)$, or $(3A)$, depending on whether it has one, two, or at least three distinct height zero degrees respectively; otherwise, if the largest degree has height one then $B$ is in class $(2B)$; if the largest degree is the sum of the other three height zero degrees then $B$ is in class $(3K)$; otherwise $B$ is in class $(3B)$. Therefore given a block of a finite group with dihedral defect groups, its ordinary character degrees can be obtained using GAP or Magma and its Morita equivalence class can always be deduced. This is sometimes, but not always, possible for a block with semidihedral defect groups.

\section{Reduction to Quasi-Simple Groups}\label{reductions}

A quasi-simple group is a perfect central extension of a simple group, i.e., a group $G$ such that $G=[G,G]$ and $G/Z(G)$ is simple. For every non-abelian finite simple group $S$, the largest quasi-simple group $G$ such that $G/Z(G)\cong S$ is called the Schur cover of $S$ and is unique up to isomorphism. 

In order to reduce our problem from all groups to quasi-simple groups, we first show a Morita equivalence for odd index normal subgroups. Note that this is not true for Klein four defect groups, as we require at least one height one character.

\begin{lemma}\label{oddindex}
Let $N\trianglelefteq G$ be finite groups with $\abs{G:N}$ odd, and let $b$ be a block of $N$ with dihedral or semidihedral defect groups (of order at least $8$). Then each block of $G$ covering $b$ is Morita equivalent to $b$.
\end{lemma}
\begin{proof}
Let $B$ be a block of $G$ covering $b$; note that these blocks have the same defect groups. By Lemma \ref{Fong1} we may assume without loss of generality that $b$ is $G$-invariant. We can also assume by the Feit-Thompson theorem that $\abs{G:N}=p$, an odd prime. 

For each $\chi\in\Irr(B)$, by Clifford's Theorem either $\chi$ restricts irreducibly to $N$ or $\chi{\downarrow}_N=\phi_1+\dots+\phi_p$ for some $\phi_1,\dots,\phi_p\in\Irr(b)$ that are conjugate in $G$; note that each constituent of $\chi{\downarrow}_N$ has the same height as $\chi$. Since the height one characters of $b$ have the same degrees, as do those of $B$, if any one of them is $G$-invariant then they all are. By Lemma \ref{inductionsimple} there is a unique $\chi\in\Irr(G)$ covering $\phi\in\Irr(b)$ if and only if $\phi$ is not $G$-invariant; in this case $\chi$ covers several characters of $b$. But $b$ and $B$ have the same number of height one characters, so all those of $b$ must be $G$-invariant and in bijection with those of $B$. Since each height one character of $b$ has more than one character of $G$ above it, there must be multiple blocks of $G$ covering $b$. Then by Lemma \ref{coveringblock} each height zero character of $b$ also has multiple characters of $G$ above it, so must also be $G$-invariant.

Therefore restriction induces a bijection between $\Irr(B)$ and $\Irr(b)$, so $B$ and $b$ are Morita equivalent by \cite[Thm 2]{Schmid}. 
\end{proof}

Now we give a reduction to quasi-simple groups for any of the classes with three simple modules, based on Brauer's \cite{Brauer} and Olsson's \cite{Olsson} analyses of blocks with dihedral and semidihedral defect groups respectively.

\begin{proposition}\label{reduction}
If $B$ is a block of a finite group $G$ with dihedral or semidihedral defect groups (of order at least $8$) and $l(B)=3$, then $B$ is Morita equivalent to a block of a finite quasi-simple group.
\end{proposition}
\begin{proof}
Let $N\trianglelefteq G$ be a maximal normal subgroup and $b$ be a block of $N$ covered by $B$. If $P$ is a defect group of $b$ then by Lemma \ref{Knorr} there is a defect group $D$ of $B$ such that $P=D\cap N$. Since $N$ is a normal subgroup, $P$ is strongly closed in $D$ with respect to $G$; that is if $x\in P$ and $g\in G$ such that $x^g\in D$, then $x^g\in P$. 

First, if $B$ has dihedral defect groups, let $X_1,X_2$ be representatives of the two conjugacy classes of Klein four-groups in $D$. Since $l(B)=3$, as in \cite[Section 4]{Brauer} case (aa), there are elements of order $3$ in $G$ that permute the non-trivial elements of $X_1$ and $X_2$ respectively. This means that the involution in $X_1\cap X_2=Z(D)$ is $G$-conjugate to elements in both non-central conjugacy classes of involutions in $D$; hence all involutions in $D$ -- of which there are more than $\abs{D}/2$ -- are $G$-conjugate. Therefore, since $P$ is strongly closed in $D$ with respect to $G$, either $P=1$ or $P=D$.

If $B$ has semidihedral defect groups, then there are three conjugacy classes of subgroups of order $4$ in $D$: one of Klein four-groups, one of cyclic groups, and another single cyclic group $C$ contained in the index-$2$ cyclic subgroup of $D$. Since $l(B)=3$, as in case (aa) of \cite{Olsson}, again there is an element of order $3$ in $G$ permuting the non-trivial elements of a Klein four-group. Hence if $\abs{P}>1$ then $P$ contains all the involutions in $D$, and so contains the index-$2$ dihedral subgroup of $D$; in particular $C\leq P$. Also, by \cite[Lemma 2.4]{Olsson} we see that $C$ is $G$-conjugate to all of the cyclic groups of order $4$ in $D$, which are therefore also in $P$; this is then enough to generate $D$.

So for both dihedral and semidihedral defect groups we have that either $P=1$ or $P=D$. If $P=1$ then by Lemma \ref{Fong2} there is a block of a central extension of $G/N$ Morita equivalent to $B$.

If $P=D$, first we can assume by Lemma \ref{Fong1} that $b$ is $G$-invariant. By Lemma \ref{evenindex}, since the defects are the same, $\abs{G:N}\neq2$, and if $\abs{G:N}$ is odd then $B$ and $b$ are Morita equivalent by Lemma \ref{oddindex}.

If $G/N$ is non-abelian simple, then define $G[b]$ as in \cite{Kulshammer} as the set of $g\in G$ such that the algebra automorphism $b\rightarrow b,\ x\mapsto g^{-1}xg$ is an inner automorphism of $b$; this is a normal subgroup of $G$ containing $N$. As in \cite[Ex. 1.2]{Linckelmann-Livesey} automorphisms of $b$ can be considered as $b$-$b$-bimodules, so (their images in $\mathrm{Out}(b)$) are elements of the Picard group of $b$, and those induced by group automorphisms are elements of $\mathcal{T}(b)$, the subgroup of the Picard group consisting of bimodules with trivial source. By \cite[Thm 1.1]{Boltje-Kessar-Linckelmann} if $\Aut(P)$ is solvable -- as is the case when $P$ is dihedral or semidihedral -- then $\mathcal{T}(b)$ is also solvable, and hence so is its subgroup $G/G[b]$. Then since $G/N$ is non-abelian simple, $G[b]=G$. Therefore $B$ and $b$ are Morita equivalent by \cite[Thm 7]{Kulshammer}.

Repeatedly taking maximal (non-central when possible) normal subgroups in this way gives a group whose only normal subgroups are central -- which is non-abelian and therefore quasi-simple -- with a block Morita equivalent to $B$.
\end{proof}

We now give two less general reductions, each for a specific semidihedral class with two simple modules. Note that each proof essentially builds on the previous ones.

\begin{proposition}\label{reduction2}
If $B$ is a block of a finite group $G$ with semidihedral defect groups in either of the $(2B_4)$ Morita equivalence classes, then $B$ is Morita equivalent to a block of a finite quasi-simple group.
\end{proposition}
\begin{proof}
Again let $N\trianglelefteq G$ be a maximal normal subgroup, let $b$ be a block of $N$ covered by $B$, with defect group $P=D\cap N$ for a defect group $D$ of $B$, and again assume by Lemma \ref{Fong1} that $b$ is $G$-invariant. Since $k(B)=2^{n-2}+3$ and $l(B)=2$ we are in case (ab) of \cite{Olsson}, and there is again an element of order $3$ in $G$ permuting the non-trivial elements of a Klein four-group. So if $\abs{P}>1$ then $P$ contains the index-$2$ dihedral subgroup of $D$. The cases $P=1$ and $P=D$ are the same as in the previous proof. 

Suppose that $1<P<D$, so $P$ is dihedral of index $2$ in $D$. Since the blocks have different defects, $\abs{G:N}$ cannot be odd. If $\abs{G:N}=2$, then as $B$ has an odd number of height one characters, one must split on restriction to $N$. As it is the sum of two distinct Brauer characters this would force both of them to split; but $l(b)\leq3$, a contradiction.

Suppose that $G/N$ is non-abelian simple. Any odd-order element $g\in G\setminus N$ does not permute the ordinary characters of $b$, by Lemma \ref{oddindex} considering $N\trianglelefteq\langle g\rangle N$. Then since $G/N$ can be generated by odd-order elements, each $\phi\in\Irr(b)$ is $G$-invariant. The unique block $B_{DN}$ of $DN$ covering $b$ has, by \cite[Ch. 10 Thm 5.10]{Karpilovsky} (see also \cite[Ex. 9.4]{Navarro}), semidihedral defect group $D$. But $\abs{DN:N}=2$ and each character of $b$ is $DN$-invariant, so must split on induction to $DN$ giving eight height zero characters of $B_{DN}$, which is impossible. Hence $P$ cannot have index $2$ in $D$.
\end{proof}

The following case is more complicated:

\begin{proposition}\label{reduction3}
If $B$ is a block of a finite group $G$ with semidihedral defect groups of order $2^n$ in either of the $(2B_1)$ Morita equivalence classes, then there is a quasi-simple group $S$ with a block $b_S$ such that either $B$ is Morita equivalent to $b_S$, or $b_S$ has dihedral defect groups in class $(3K)$ and is covered by a block of a group $S.2$ with semidihedral defect groups of order $2^n$ also in one of the $(2B_1)$ classes.
\end{proposition}

\begin{proof}
Again let $N\trianglelefteq G$ be a maximal normal subgroup, let $b$ be a block of $N$ covered by $B$, with defect group $P=D\cap N$ for a defect group $D$ of $B$, and assume by Lemma \ref{Fong1} that $b$ is $G$-invariant. As in the previous proof, $P$ is either $1$, $D$, or dihedral of index $2$ in $D$, and the first two cases are as in the proof of Proposition \ref{reduction}.

Suppose then that $P$ is dihedral of index $2$ in $D$. By the same arguments as in the previous proof, $G/N$ cannot be non-abelian simple or of odd order, so $\abs{G:N}=2$ and $G=DN$. Then one of the height one characters of $B$, and hence also the second Brauer character, must split on restriction to $N$, giving that $b$ is in dihedral class $(3K)$.

Now let $H\trianglelefteq N$, and first suppose in addition that $H\trianglelefteq G$. Let $b_H$ be a block of $H$ covered by $b$ with defect group $P\cap H$; as in the proof of Proposition \ref{reduction} this is either $1$ or $P$. If $P\cap H=1$ then $D\cap H=1$ or $2$. But $D\cap H$ must also be $1$, $D$, or index $2$ in $D$, so $D\cap H=1$, and thus $B$ is Morita equivalent to a block of a central extension of $G/H$ by Lemma \ref{Fong2}. If $P\cap H=P$ then $b_H$ is Morita equivalent to $b$, as in the proof of Proposition \ref{reduction}. Then the unique block $B_{DH}$ of $DH$ covering $b_H$ has defect group $D$ and, as with $b$, the second and third Brauer characters of $b_H$ must fuse on induction to $DH$, so $B_{DH}$ is also in one of the $(2B_1)$ classes.

Using the above we now assume that the only normal subgroups $H\trianglelefteq N$ that are also normal in $G$ are $N$ and those contained in $Z(G)$; then $N/Z(G)$ must be a direct product of simple groups. Let $H\trianglelefteq N$ be such that $H/Z(G)$ is simple, and suppose that $H\neq N$. Then $N$ is a central product of the $G$-conjugates of $H$, which are $H$ and $H^x$, where $x\in D\setminus N$.

Let $b_H$ again be the block of $H$ covered by $b$. The block $(b_H)^x$ of $H^x$ conjugate to $b_H$ is covered by $B$, and $b$ is the unique block of $N$ covered by $B$, so $(b_H)^x$ is the unique block of $H^x$ covered by $b$, and $b$ is a central product of $b_H$ and $(b_H)^x$. Since the centre of $G$ must be odd, the defect groups of $b$ are a direct product of the defect groups of $b_H$ and $(b_H)^x$; they are dihedral, so must be a product of a dihedral and trivial group. But $b_H$ and $(b_H)^x$ are conjugate in $G$, so must have isomorphic defect groups, a contradiction. Therefore no such $H\neq N$ exists and $N$ is quasi-simple.
\end{proof}

In Section \ref{2B1} we show that there are no blocks in the $(2B_1)$ classes with defect at least $5$ for groups with quasi-simple subgroups of index $2$, as in the latter half of Proposition \ref{reduction3}. Hence if $\abs{D}\geq32$ then $P$ cannot have index $2$ in $D$.

Other than for $(2B_1)$, it is sufficient to show that there are no blocks of quasi-simple groups in the previously mentioned Morita equivalence classes to show that there are no blocks of any finite groups in those classes. In fact, since for any of these classes the non-trivial central element of $D$ is $G$-conjugate to non-central elements, the centre of $G$ must be odd; hence we need only consider odd covers of the finite simple groups.

We highlight several facts from the previous proofs for future reference:

\begin{remark}\label{reductionremark}
Let $B$ be a block of a finite group $G$ with dihedral or semidihedral defect groups and three simple modules, or semidihedral defect groups in one of the $(2B_4)$ or $(2B_1)$ classes. Then:
\begin{enumerate}[label=(\roman*)]
    \item $\abs{Z(G)}$ is odd.
\end{enumerate}

Let $b$ be a block of a normal subgroup $N\trianglelefteq G$ covered by $B$, and let $P$ and $D$ be defect groups of $b$ and $B$ respectively such that $P=D\cap N$.
\begin{enumerate}[label=(\roman*)]
\setcounter{enumi}{1}
    \item Unless $B$ is in class $(2B_1)$ with $\abs{D}=16$, either $P=1$ or $P=D$.
    \item If $P=D$ then $b$ and $B$ are Morita equivalent.
\end{enumerate}
\end{remark}

\subsection{Reduction to Quasi-Isolated Blocks}

Let $B$ be a block of a quasi-simple group of Lie type $G$ in cross characteristic with dihedral or semidihedral defect groups and three simple modules, semidihedral defect groups in one of the $(2B_4)$ classes, or semidihedral defect groups of order at least $32$ in one of the $(2B_1)$ classes. Let $\mathbf{G}$ be the corresponding simply connected simple algebraic group, with Steinberg endomorphism $F$ so that $G=\mathbf{G}^F/Z$, where $Z$ is a central subgroup; we may assume that $\abs{Z}$ is a power of $2$. If $B$ is not quasi-isolated, then by Theorem \ref{BonnafeRouquiercentral} it is Morita equivalent to a quasi-isolated block $B_L$ of some $L=\mathbf{L}^F/Z$, where $\mathbf{L}$ is a proper Levi subgroup of $\mathbf{G}$.

By the structure of $L$ as a central product of groups of Lie type and tori, $B_L$ is a central product of blocks, and each defect group $D$ is a central product. Since $D$ is non-abelian, it is not contained in a torus, so there must be some quasi-simple normal subgroup $S$ of $L$ that intersects $D$ non-trivially. Therefore Remark \ref{reductionremark} implies that $S$ contains $D$, and that there is a block $B_S$ of $S$ covered by $B_L$ that is Morita equivalent to $B_L$.

The rank of $S$ as a group of Lie type is at least $1$ but strictly less than the rank of $G$; hence if $G$ is of rank $1$ then $B$ must be quasi-isolated. Therefore induction on the rank of $G$ gives that $B$ is Morita equivalent to a quasi-isolated block of a quasi-simple group of Lie type. Hence to identify which Morita equivalence classes occur among finite groups of Lie type in cross characteristic it is sufficient to consider only the quasi-isolated blocks of the quasi-simple groups.

\section{Blocks of Quasi-Simple Groups}\label{proof}

We go through the finite simple groups and their odd covers case by case, and study the blocks with dihedral or semidihedral defect groups. 

\subsection{Defining Characteristic}

First consider the defining characteristic case, that is where $G$ is a finite group of Lie type of characteristic $p=\ell=2$. Humphreys \cite{Humphreys} proved that every $p$-block of $G$ has defect groups either trivial or the Sylow $p$-subgroups of $G$. The groups with dihedral Sylow $2$-subgroups were classified by Gorenstein and Walter \cite{Gorenstein-Walter}; in particular any finite quasi-simple group with dihedral Sylow $2$-subgroups is isomorphic to an odd cover of either $\Alt(7)$ or $\PSL_2(q)$ for $q$ odd. Similarly, by a result of Alperin, Brauer, and Gorenstein \cite{Alperin-Brauer-Gorenstein} any quasi-simple group with semidihedral Sylow $2$-subgroups is isomorphic to $M_{11}$ or an odd cover of $\PSL_3(q)$ for $q\equiv-1\bmod4$ or $\PSU_3(q)$ for $q\equiv1\bmod4$. Hence we need not further consider blocks in defining characteristic; note that while $\PSL_3(2)$ has dihedral Sylow $2$-subgroups it is isomorphic to $\PSL_2(7)$, which will be considered later.

\subsection{Types B, C, D}\label{classical}

Now let $\mathbf{G}$ be a simple simply connected group of Lie type of characteristic $p\neq\ell=2$, with $F$ a Steinberg endomorphism so that $\mathbf{G}^F$ is a finite group of Lie type, and let $G=\mathbf{G}^F/Z$, where $Z$ is a central subgroup of $\mathbf{G}^F$. First let $\mathbf{G}$ be of type $B_n$, $C_n$, or $D_n$ for $n>1$. Then $\mathbf{G}$ has no non-trivial quasi-isolated $2^\prime$-elements \cite[Table II]{Bonnafe}, and the only unipotent $2$-blocks of $\mathbf{G}^F$ are the principal blocks \cite[Prop. 6]{Enguehard}. Hence the only quasi-isolated block of $G$ is the principal block, which has defect groups the Sylow $2$-subgroups of $G$, which are neither dihedral nor semidihedral for any such $G$. Note that this includes the twisted groups $^2\!D_n(q)$ and $^3\!D_4(q)$, while $^2\!B_2(q)$ only exists in characteristic $2$.

\subsection{Exceptional Groups}

\subsubsection{Unipotent Blocks}\label{E7}

Enguehard \cite{Enguehard} classified all unipotent blocks of finite groups of Lie type for bad primes, as $2$ is other than for type $A$. There are none with semidihedral defect groups, and only $E_7(q)$ has unipotent blocks with dihedral defect groups. 

\begin{theorem}\label{E7theorem}
The simple group $E_7(q)$ for $q\equiv1\bmod4$ (resp. $-1\bmod4$) has two unipotent blocks with dihedral defect groups of order $\abs{q-1}_2$ (resp. $\abs{q+1}_2$) in class $(3A)$ (resp. $(3K)$). 
\end{theorem}

\begin{proof}
Let $G=E_7(q)_{ad}$, so $[G,G]=E_7(q)$ and $\abs{G:[G,G]}=2$. By \cite[Section 3.2]{Enguehard} there are two non-principal unipotent blocks of $G$ with dihedral defect groups of order $2\abs{q\pm1}_2$, corresponding to unipotent characters labelled $E_6[\theta]$ and $E_6[\theta^2]$ of a Levi subgroup $E_6(q).(q-1)$ of $G$. Let $B$ be one of these blocks, which contains the unipotent characters denoted in \cite[Section 13.9]{Carter} by $E_6[\theta^i],1$ and $E_6[\theta^i],\varepsilon$, for $i=1$ or $2$. These two unipotent characters of $G$ restrict irreducibly to $[G,G]$, so $B$ covers a unique block $b$ of $[G,G]$ containing these two unipotent characters, which each split on induction up to $G$ into pairs of irreducible characters $\chi_1,\chi_1^\prime$ and $\chi_\varepsilon,\chi_\varepsilon^\prime$ respectively in $B$.

From \cite[Section 13.9]{Carter} the degrees of these characters are related by $\chi_\varepsilon(1)=q^9\chi_1(1)$. Since $\chi_1,\chi_\varepsilon$ have the same height but different degrees, and $B$ has dihedral defect groups, they both have height zero, as do $\chi_1^\prime$ and $\chi_\varepsilon^\prime$. Then, since among these four height zero characters there are exactly two distinct degrees, $B$ must be in class $(2A)$ or $(2B)$. Since $B$ has defect at least $3$, it contains at least one height one character, whose degree is $\chi_\varepsilon(1)+\chi_1(1)=(q^9+1)\chi_1(1)$ for $(2A)$, which must therefore be the case if $q\equiv1\bmod4$, and $\chi_\varepsilon(1)-\chi_1(1)=(q^9-1)\chi_1(1)$ for $(2B)$ if $q\equiv-1\bmod4$.

Note that the defect groups of $b$ are also dihedral (or Klein four), of order $\abs{q\pm1}_2$. Indeed by Lemma \ref{evenindex} they are index-$2$ subgroups of defect groups of $B$, so supposing they are not dihedral they must be cyclic, so have only height zero characters by \cite[Thm 1.1]{Kessar-Malle}. So then all of the height one characters of $B$ must split on restriction to $[G,G]$, which would lead to either none of the rows or two distinct rows of the decomposition matrix of $b$ being repeated. But since the defect groups are of order at least $4$, there must be exactly one repeated row as in \cite[Thm 5.1.2]{Craven}.

There are an odd number of them, so one of the height one characters of $B$ must split on restriction to $[G,G]$, and thus so must the second Brauer character. Then all the other ordinary characters must fuse, and thus from the resulting decomposition matrix $b$ must be in class $(3A)$ if $q\equiv1\bmod4$, and $(3K)$ if $q\equiv-1\bmod4$.
\end{proof}

Note from the proof that $E_7(q)_{ad}$ has unipotent blocks with dihedral defect groups in class $(2A)$ or $(2B)$ for $q\equiv1$ or $-1\bmod4$ respectively.

\subsubsection{Quasi-Isolated Blocks}

Kessar and Malle \cite{Kessar-Malle} described the non-unipotent quasi-isolated blocks of exceptional groups of Lie type for bad primes $\ell\neq p$. Let $\mathbf{G}$ be a simply connected simple exceptional algebraic group, with $F$ a Steinberg endomorphism so that $\mathbf{G}^F$ is a finite group of Lie type, and $\mathbf{G}^*$ is a group dual to $\mathbf{G}$. They list all quasi-isolated $\ell^\prime$-elements $s\in\mathbf{G}^*$ and describe the quasi-isolated $\ell$-blocks in each Lusztig series $\mathcal{E}(\mathbf{G}^F,s)$, in particular proving Theorem \ref{Kessar-Malle}. For $\ell=2$ in most cases the corresponding $Z(\mathbf{L})_2^F$, so also the defect groups of the block, contains a $4\times4$ subgroup, so they are neither dihedral nor semidihedral; nor are the defect groups for $E_7(q)$ after quotienting by the simply connected group's centre of order $2$. The exceptions are the blocks of $E_6(q)$ and $^2\!E_6(q)$ in line $8$ of \cite[Table 3]{Kessar-Malle} and the blocks of $E_8(q)$ and $G_2(q)$ in line $6$ of \cite[Table 5]{Kessar-Malle} and line $2$ of \cite[Table 9]{Kessar-Malle} respectively. We show that the defect groups of these blocks cannot be dihedral, and deduce the possible classes of the block if they are semidihedral. Note that for each of these exceptions the blocks are isolated and the centraliser $C_{\mathbf{G}^*}(s)^F$ has a unipotent block with semidihedral defect groups. 

Hiss and Shamash \cite{Hiss-Shamash} described the $2$-blocks of $G_2(q)$ for $q$ odd, including character degrees, defect groups, and decomposition matrices. The case in question only occurs when $q\equiv\pm5\bmod12$. By \cite[(2.2.3)]{Hiss-Shamash}, if $q\equiv5\bmod12$ then $G_2(q)$ has a block with semidihedral defect groups of order $4\abs{q-1}_2$ and decomposition matrix that of class $(3A_1)$. This is the block corresponding to line $2$ in \cite[Table 9]{Kessar-Malle}, since we know that block contains the character corresponding to the trivial character of $C_{\mathbf{G}^*}(s)^F={}^2\!A_2(q)$, whose degree is $\abs{G_2(q):{}^2\!A_2(q)_{ad}}_{p^\prime}=q^3-1$. Similarly by \cite[(2.3.2)]{Hiss-Shamash}, replacing $^2\!A_2(q)$ with $A_2(q)$, if $q\equiv-5\bmod12$ then $G_2(q)$ has an isolated block with semidihedral defect groups of order $4\abs{q+1}_2$ and decomposition matrix that of class $(3B_1)$ or $(3D)$.

\begin{theorem}
If $B$ is a quasi-isolated block of $E_6(q)_{sc}$ with semidihedral defect groups, then $q\equiv-1\bmod4$, the defect groups are of order $4\abs{q+1}_2$, and $B$ is in class $(3B_1)$ or $(3D)$. If $B$ is a quasi-isolated block of $^2\!E_6(q)_{sc}$ with semidihedral defect groups, then $q\equiv1\bmod4$, the defect groups are of order $4\abs{q-1}_2$, and $B$ is in class $(3A_1)$ or $(3C_{2,2})$.
\end{theorem}

\begin{proof}
First let $\mathbf{G}^F=E_6(q)_{sc}$. Block $8$ in \cite[Table 3]{Kessar-Malle} for $q\equiv-1\bmod4$ has $C_{\mathbf{G}^*}(s)^F=A_2(q^3).3$ and defect groups of the form $(q-1).2.2$. It is the only block in $\mathcal{E}(\mathbf{G}^F,s)$, so it contains the Jordan correspondents of all unipotent characters of $C_{\mathbf{G}^*}(s)^F$, denoted $\phi_3$, $\phi_{21}$, and $\phi_{1^3}$ according to the partitions of $3$. Note that $C_{\mathbf{G}^*}(s)$ is disconnected, so the unipotent characters are those of $(C_{\mathbf{G}^*}(s)^\circ)^F=A_2(q^3)$ induced up to $C_{\mathbf{G}^*}(s)^F$, but these characters are fixed by the field automorphism. 

The degrees of the three known characters are $\phi_3(1)=1$, $\phi_{21}(1)=q(q+1)$, and $\phi_{1^3}(1)=q^3$ each multiplied by a factor of $\abs{\mathbf{G}^F:C_\mathbf{G}(s)^F}_{p^\prime}$. Each character of the block must have this factor, so the $\phi_3$ character has height zero. Then the $\phi_{21}$ character has height at least $2$, so the defect groups of the block are not dihedral. Suppose they are semidihedral. From the three known character degrees we can rule out the decomposition matrices of $(3C_{2,1})$ and $(3B_2)$ by the large height character not having largest degree, and $(3A_1)$ and $(3C_{2,2})$ as this would force a contradiction either on the linear dependence of these three characters or the repeated row having height greater than one. For $(3H)$, the configurations not giving either of the previous contradictions has the $\phi_3$ character as the third row of the decomposition matrix, giving the second row degree $\abs{\mathbf{G}^F:C_\mathbf{G}(s)^F}_{p^\prime}\cdot(q^2+q-1)$, which notably is not a cyclotomic polynomial. Using GAP and \cite{Lubeck}, we see that there are no ordinary character degrees of $E_6(q)_{sc}$ which coincide with this degree for any $q$. (This was done by listing all character degrees as polynomials in $q$, subtracting the polynomial degree $\abs{\mathbf{G}^F:C_\mathbf{G}(s)^F}_{p^\prime}\cdot(q^2+q-1)$, and checking for any positive integer roots up to a sufficiently high $q$; a polynomial's roots are bounded by twice the absolute value of the largest coefficient divided by its leading coefficient.) Therefore the block is in class $(3B_1)$ or $(3D)$; these classes have the same decomposition matrix.

Similarly, by Ennola duality and using the same arguments, $\mathbf{G}^F={}^2\!E_6(q)_{sc}$ with $q\equiv1\bmod4$ has a block with $C_{\mathbf{G}^*}(s)^F={}^2\!A_2(q^3).3$ that may have semidihedral defect groups, and if so is in class $(3A_1)$ or $(3C_{2,2})$, whose decomposition matrices are indistinguishable based only on ordinary character degrees.
\end{proof}

\begin{theorem}
If $B$ is an isolated block of $E_8(q)$ with semidihedral defect groups, then either $q\equiv5\bmod12$, the defect groups are of order $\abs{q-1}_2$, and $B$ is in class $(3A_1)$ or $(3C_{2,2})$; or $q\equiv7\bmod12$, the defect groups are of order $\abs{q+1}_2$, and $B$ is in class $(3B_1)$ or $(3D)$.
\end{theorem}

\begin{proof}
Let $\mathbf{G}^F=E_8(q)$; note that $\mathbf{G}^*\cong\mathbf{G}$. When $q\equiv1\bmod4$, block $6$ of \cite[Table 5]{Kessar-Malle} occurs when also $q\equiv-1\bmod3$, and has $C_\mathbf{G}(s)^F={}^2\!E_6(q).{}^2\!A_2(q)$; here $\mathcal{E}_2(\mathbf{G}^F,s)$ is a union of three blocks. Two of these blocks have defect groups of the form $P$ with subgroups $A\leq D\leq P$ such that $D\cong2\abs{q-1}_2$ is the unique cyclic subgroup of index $2$ in $P$, and any $\sigma\in P\setminus D$ inverts $A\cong\abs{q-1}_2$; this means that $P$ is dihedral, semidihedral, or generalised quaternion, of order $4\abs{q-1}_2$. These two blocks, corresponding to $i=1$ and $2$, each contain two Harish-Chandra series: one series has $\mathbf{L}^F=\Phi_1.E_7(q)$, $\lambda={}^2\!E_6[\theta^i]$, and $W_{\mathbf{G}^F}(\mathbf{L},\lambda)=A_1=2$; and the other has $\mathbf{L}^F=E_8(q)$, $\lambda={}^2\!E_6[\theta^i]\otimes\phi_{21}$, and $W_{\mathbf{G}^F}(\mathbf{L},\lambda)=1$ (here $\lambda$ is labelled by the Jordan corresponding unipotent character in $C_\mathbf{G}(s)^F$). The latter series consists only of the character $^2\!E_6[\theta^i]\otimes\phi_{21}$, since $\mathbf{L}^F=\mathbf{G}^F$, and the former contains $^2\!E_6[\theta^i]$ and one other character, since $W_{\mathbf{G}^F}(\mathbf{L},\lambda)$ has order $2$ and hence $R_\mathbf{L}^\mathbf{G}(\lambda)$ has norm $2$. Assume $i=1$ (all degrees are identical for $i=2$) and let $B$ be the corresponding block of $\mathbf{G}^F$.

Each character of $\Irr(B)$ corresponds to an ordinary character of $^2\!E_6(q).{}^2\!A_2(q)$ and has $\abs{\mathbf{G}^F:C_\mathbf{G}(s)^F}_{p^\prime}=\abs{E_8(q):{}^2\!E_6(q).{}^2\!A_2(q)}_{p^\prime}$ as a common factor; in particular, the $2$-part of this is $2^3\abs{(q-1)^3}_2$. The $2$-part of $E_8(q)$ is $2^{14}\abs{(q-1)^8}_2$ and the defect groups of $B$ have order $2^2\abs{q-1}_2$. Then since $\abs{{}^2\!E_6[\theta](1)}_2=2^9\abs{(q-1)^4}_2$, the correspondent of $^2\!E_6[\theta]$ has height zero and that of $^2\!E_6[\theta]\otimes\phi_{21}$ has height at least $2$, since $\phi_{21}(1)=q(q-1)$; therefore the defect groups are not dihedral.

Of all the other unipotent characters of $C_\mathbf{G}(s)^F$, found in \cite[Section 13.9]{Carter}, only $^2\!E_6[\theta]\otimes\phi_{1^3}$ has a large enough $2$-part given the defect of $B$ -- except also perhaps $^2\!E_6[1]\otimes\phi_{21}$, but \cite[Table 5]{Kessar-Malle} lists this character in a Harish-Chandra series of the other of the three blocks contained in $\mathcal{E}_2(\mathbf{G}^F,s)$ -- and it has height zero since $\phi_{1^3}(1)=q^3$.

Suppose $B$ has semidihedral defect groups. As in the previous proof, using the three known character degrees the decomposition matrices of $(3C_{2,1})$, and $(3B_2)$ can be ruled out by the large height character not having largest degree, and $(3B_1)$ and $(3D)$ by linear dependencies and necessary heights. Similarly for $(3H)$, the third row is forced to have degree $\abs{\mathbf{G}^F:C_\mathbf{G}(s)^F}_{p^\prime}\cdot{}^2\!E_6[\theta](1)\cdot(q^2-q-1)$, again not a cyclotomic polynomial, which does not coincide with any ordinary character degree of $^2\!E_6(q).{}^2\!A_2(q)$ for any $q$ (checking all products of degrees of $^2\!E_6(q)$ with those of $^2\!A_2(q)$ up to sufficiently large $q$). Therefore the block is in class $(3A_1)$ or $(3C_{2,2})$.

If $q\equiv-1\bmod4$ then by Ennola duality and using the same arguments, with $C_\mathbf{G}(s)^F=E_6(q).A_2(q)$, we get that $B$ is in class $(3B_1)$ or $(3D)$.
\end{proof}

Note that $^2\!F_4(q)$ only exists in characteristic $2$, and $^2G_2(q)$ has elementary abelian Sylow $2$-subgroups.

\subsection{Type A}\label{typeA}

Now consider the groups of type $A$, that is the linear and unitary groups. The general unitary group will be denoted $\GL_n(-q)$ or $\GU_n(q)$ (as opposed to $\GU_n(q^2)$), and similarly for the special and projective unitary groups. Let $b$ be a $2$-block of an odd cover of $\PSL_n(\varepsilon q)$, where $\varepsilon=\pm1$, contained in a block $b_{SL}$ of $\SL_n(\varepsilon q)$, with semisimple label $s\in\PGL_n(\varepsilon q)$ according to Lusztig series, and let $B$ be a block of $G=\GL_n(\varepsilon q)$ covering $b_{SL}$. The defect groups of the blocks of the general linear and unitary groups were described by Brou\'e in \cite[(3.7)]{Broue}, and those of $B$ are the Sylow $2$-subgroups of $C_G(\tilde{s})\cong\prod_{i\in I}\GL_{n_i}\left((\varepsilon q)^{a_i}\right)$ for some suitable $n_i$ and $a_i$, where $\tilde{s}$ is a representative of $s$ in $G$. 

Suppose that $b$ is quasi-isolated and has dihedral or semidihedral defect groups. By \cite[Table II]{Bonnafe} quasi-isolated semisimple elements $s\in\PGL_n(k)$ have connected centraliser of the form $(\PGL_{n/d}(k))^d$, where $d$ is the order of $s$ which is therefore odd. Therefore $C_G(\tilde{s})\cong\prod_{i\in I}\GL_{n/d}\left((\varepsilon q)^{a_i}\right)$.

We must have $n/d>1$, otherwise the defect groups of $B$ would be abelian. Additionally if $\abs{I}>1$ then $C_G(\tilde{s})$ would at least contain $\GL_2\left((\varepsilon q)^{a_1}\right)\times\GL_2\left((\varepsilon q)^{a_2}\right)$ as a subgroup, so the defect groups of $b_{SL}$ would contain the Sylow $2$-subgroups of $\SL_2\left((\varepsilon q)^{a_1}\right)\times\SL_2\left((\varepsilon q)^{a_2}\right)$ -- a product of two generalised quaternion groups -- which after quotienting by a cyclic subgroup cannot be dihedral or semidihedral.

Therefore $C_G(\tilde{s})\cong\GL_{n/d}\left((\varepsilon q)^d\right)=\GL_{n/d}(\varepsilon q^d)$; let $\sigma:\GL_{n/d}(\varepsilon q^d)\rightarrow\GL_n(\varepsilon q)$ denote an embedding. For any $x$ we have $\det_n(x\sigma)=\left(\det_{n/d}(x)\right)^m$, where $m=(q^d-\varepsilon)/(q-\varepsilon)$ which is odd since $d$ is, so if $x$ is a $2$-element then $\det_n(x\sigma)=1$ precisely when $\det_{n/d}(x)=1$. Hence, with the following subscript-$2$ notation denoting a Sylow $2$-subgroup, $\left(\GL_{n/d}(\varepsilon q^d)_2\right)\sigma\cap \SL_n(\varepsilon q)=\left(\SL_{n/d}(\varepsilon q^d)_2\right)\sigma$, so the defect groups of $b_{SL}$ are isomorphic to $\SL_{n/d}(\varepsilon q^d)_2$, and those of $b$ are isomorphic to $\PSL_{n/d}(\varepsilon q^d)_2$. In particular the defect groups of $b$ are dihedral (or Klein four) if and only if $n/d=2$, and semidihedral if and only if $n/d=3$ and $q\equiv-\varepsilon\bmod4$.

We first look at the semidihedral case, as it is more straightforward.

\begin{theorem}
If $b$ is a quasi-isolated block of (a non-exceptional odd cover of) $\PSL_n(q)$ (resp. $\PSU_n(q)$) with semidihedral defect groups, then $q\equiv-1\bmod4$ (resp. $q\equiv1\bmod4$), the defect groups are of order $4\abs{q+1}_2$ (resp. $4\abs{q-1}_2$), and $b$ is in class $(3B_1)$ or $(3D)$ (resp. $(3A_1)$).
\end{theorem}

\begin{proof}
Using the previous notation of this section, we have that $C_G(\tilde{s})=\GL_3(\varepsilon q^{n/3})$, with $n/3$ odd and $q^{n/3}\equiv q\equiv-\varepsilon\bmod4$. If $n=3$ then $b$ is the principal block as in Theorem \ref{semidihedral}, so assume $n>3$. The decomposition matrices of the principal blocks of $\PSL_3(q^{n/3})$ for $q\equiv-1\bmod4$ and $\PSU_3(q^{n/3})$ for $q\equiv1\bmod4$ are
\begin{equation*}
\begin{pmatrix}
1&.&.\\1&1&.\\1&.&1\\1&1&1\\.&.&1\\.&1&.\\
\end{pmatrix}
\quad\text{and}\quad
\begin{pmatrix}
1&.&.\\1&1&.\\1&.&1\\1&1&1\\.&.&1\\2&1&1\\
\end{pmatrix}
\end{equation*}
respectively, with the last rows repeated $\abs{q+\varepsilon}_2-1$ times. Since $\abs{Z(\SL_3(\varepsilon q^{n/3}))}$ is odd, the principal block of $\SL_3(\varepsilon q^{n/3})$ is isomorphic to that of $\PSL_3(\varepsilon q^{n/3})$. Define $H^\prime$ to be the subgroup of $\GL_3(\varepsilon q^{n/3})$ consisting of elements whose determinant has odd order in $\mathbb{F}_{q^{n/3}}^\times$. Then since $\abs{q^{n/3}-\varepsilon}_2=2$ we have that $\abs{H^\prime:\SL_3(\varepsilon q^{n/3})}$ is odd, so by Lemma \ref{oddindex} the principal block of $H^\prime$ is Morita equivalent to that of $\SL_3(\varepsilon q^{n/3})$, and $\GL_3(\varepsilon q^{n/3})=H^\prime\times2$, so the principal block of $\GL_3(\varepsilon q^{n/3})$ has the corresponding decomposition matrix above with each row occurring twice.

Similarly, since $n$ is odd and $\abs{q-\varepsilon}_2=2$, if $H$ is the subgroup of $G=\GL_n(\varepsilon q)$ of elements with determinant of odd order then $G=H\times2$ and $b$ is Morita equivalent to the block $B_H$ of $H$ covered by the block $B$ of $G$.

Since $C_{\mathbf{G}}(\tilde{s})$ must be a Levi subgroup of $\mathbf{G}=\GL_n(\overline{\mathbb{F}_q})$, Theorem \ref{Bonnafe-Rouquier} implies that $B$ is Morita equivalent to the principal block of $C_G(\tilde{s})=\GL_3(\varepsilon q^{n/3})$ with decomposition matrix as described above. Then since $G=H\times2$, the ordinary characters of $B$ restrict irreducibly in pairs to those of $B_H$, so $B_H$, and hence also $b$, has decomposition matrix that of the principal block of $\PSL_3(\varepsilon q^{n/3})$. Therefore $b$ is in class $(3B_1)$ or $(3D)$ if $\varepsilon=1$, and is in class $(3A_1)$ if $\varepsilon=-1$.
\end{proof}

For the dihedral case we first calculate a decomposition matrix we will use in the proof. Throughout the rest of this section set $a=\abs{q+1}_2$ and $c=\abs{q-1}_2$. 

\begin{proposition}\label{gu2}
The decomposition matrix of the principal block $B$ of $\GU_2(q)$ is
\begin{equation*}
\begin{blockarray}{c(cc)c}
    1&1&.&a\text{ times}\\
    q&1&1&a\text{ times}\\
    q-1&.&1&\frac{1}{2}a\,(a-1)\text{ times}\\
    q+1&2&1&\frac{1}{2}a\,(c-1)\text{ times}\\
\end{blockarray};
\end{equation*}
the character degrees are shown on the left, and each row is repeated the number of times shown on the right.
\end{proposition}

\begin{proof}
Note that $\SL_2(q)\cong \SU_2(q)$. First let $q\equiv-1\bmod4$. Then as in Corollary \ref{quaternion}, the decomposition matrix of the principal block $b$ of $\SU_2(q)$ is
\begin{equation*}
    \begin{blockarray}{(ccc)}
        1&.&.\\
        .&1&.\\
        .&.&1\\
        1&1&1\\
        1&1&.\\
        1&.&1\\
        .&1&1
    \end{blockarray}
\end{equation*}
with the final row repeated $a-1$ times. The central product $A=\SU_2(q)\ast(q+1)$ is an index-$2$ subgroup of $\GU_2(q)$, and the decomposition matrix of the principal block $B_A$ of $A$ is that of $b$ with each row occurring $a/2$ times.

By \cite[(4.3)]{Erdmann2} $l(B)=2$ and $k(B)=a^2/2+2a$. Therefore, on induction to $\GU_2(q)$, the second and third Brauer characters of $B_A$ must fuse, which forces the ordinary characters relating to the second and third rows and also the fifth and sixth rows above to fuse in pairs, while all the other ordinary characters must split to give the correct value for $k(B)$. Hence the decomposition matrix of $B$ is as claimed. 

If $q\equiv1\bmod4$ then the argument is identical with a slightly different matrix for $\SU_2(q)$; note also that then $B$ is as in Theorem \ref{semidihedral2}.
\end{proof}

\begin{theorem}\label{typeAdihedral}
If $b$ is a quasi-isolated block of (a non-exceptional odd cover of) $\PSL_n(q)$ or $\PSU_n(q)$ with dihedral defect groups (of order at least $8$), then $n/2$ is odd and either $q\equiv1\bmod4$, the defect groups are of order $\abs{q-1}_2$, and $b$ is in class (3A); or $q\equiv-1\bmod4$, the defect groups are of order $\abs{q+1}_2$, and $b$ is in class (3K).
\end{theorem}

\begin{proof}
Using the previous notation of this section, by those considerations, $C_G(\tilde{s})=\GL_2(\varepsilon q^{n/2})$ with $n/2$ odd. If $n=2$ then $b$ is the principal block as in Theorem \ref{classification}, so assume $n>2$. 

Consider $b$ as a block of $\SL_n(\varepsilon q)/2$, an odd cover of $\PSL_n(\varepsilon q)$. Set $H$ as the subgroup of $G=\GL_n(\varepsilon q)$ consisting of elements with odd-order determinant, let $B_{H/2}$ be a block of $H/2$ covering $b$, which by Lemma \ref{oddindex} is Morita equivalent to $b$, and let $B_H$ be the block of $H$ containing $B_{H/2}$ (choose $B_{H/2}$ so that the block $B$ of $G$ covers $B_H$). The central product $A=H\ast(q-\varepsilon)$ is an index-$2$ subgroup of $G$, and $\abs{A:H}=\abs{q-\varepsilon}_2/2$. Let $B_A$ be the unique block of $A$ covering $B_H$, so $B$ is the unique block of $G$ covering $B_A$.

Since $C_{\mathbf{G}}(\tilde{s})$ must be a Levi subgroup of $\mathbf{G}=\GL_n(\overline{\mathbb{F}_q})$, Theorem \ref{Bonnafe-Rouquier} implies that $B$ is Morita equivalent to the principal block of $C_G(\tilde{s})=\GL_2(\varepsilon q^{n/2})$ with the same decomposition matrix. The decomposition matrix of the principal block of $\GU_2(q^{n/2})$ is as in Proposition \ref{gu2}, replacing $q$ with $t=q^{n/2}$, and that of $\GL_2(q^{n/2})$, calculated using \cite{James}, is
\begin{equation*}
\begin{blockarray}{c(cc)c}
    1&1&.&c\text{ times}\\
    t&1&1&c\text{ times}\\
    t-1&.&1&\frac{1}{2}c\,(a-1)\text{ times}\\
    t+1&2&1&\frac{1}{2}c\,(c-1)\text{ times}\\
\end{blockarray};
\end{equation*}
note that since $n/2$ is odd, $\abs{t\pm1}_2=\abs{q\pm1}_2$. 

The decomposition matrix of $B_A$ is that of $B_H$ with each row repeated $a/2$ or $c/2$ times, depending on $q\bmod4$; the characters of $B_{H/2}$ are then a subset of those of $B_H$, and $B_{H/2}$ has dihedral defect groups. Since the defects of $B$ and $B_A$ are different, all the height zero characters of $B$ must fuse on restriction to $B_A$, giving two possible height zero characters of $B_{H/2}$. Then one of the height one characters, and hence the second Brauer character, of $B$ must split on restriction to $A$. It follows that the decomposition matrix of $B_{H/2}$, hence that of $b$, must be that of class $(3A)$ if $q\equiv1\bmod4$ or $(3K)$ if $q\equiv-1\bmod4$.
\end{proof}

\subsection{Exceptional Schur Covers}

The following are the maximal odd covers of finite simple groups of Lie type that are not quotients of the corresponding simply connected group: $3\cdot\PSL_2(9)$, $3^2\cdot\PSU_4(3)$, $3\cdot B_3(3)$, and $3\cdot G_2(3)$. The decomposition matrices of these groups are all in the Modular Atlas \cite{ModularAtlas}. The new $2$-blocks occurring -- that is those that are not blocks of the simple group -- for $3^2\cdot\PSU_4(3)$ and $3\cdot G_2(3)$ all have either maximal or zero defect, while $3\cdot B_3(3)=3\cdot O_7(3)$ has two blocks with dihedral defect groups of order $8$ in class $(2A)$. Finally, $3\cdot\PSL_2(9)\cong3\cdot\Alt(6)$ is considered in the following section.

\subsection{Alternating Groups}

The alternating groups are investigated via the representation theory of the symmetric groups, which is well understood (see \cite{James-Kerber} for example). The ordinary characters of $\Sym(n)$ correspond to the partitions of $n$, and the process of repeatedly removing $\ell$-hooks -- partitions of the form $(\ell-m,1^m)$ for some $m$ -- from a partition results in a uniquely defined $\ell$-core. Two characters are in the same $\ell$-block if and only if the $\ell$-cores of their partitions are the same, so the blocks are labelled by $\ell$-cores. The weight of a character, and thus the weight of its block, is the number of $\ell$-hooks removed to get to its $\ell$-core. The defect groups of a block of weight $w$ are then conjugate to Sylow $p$-subgroups of $\Sym(\ell w)$. 

\begin{theorem}
If $b$ is a block of an alternating group $\Alt(n)$ with dihedral defect groups (of order at least $8$), then they must be of order $8$ and either:
\begin{enumerate}[label=(\roman*)]
\item $n=6$ and $b$ is in class $(3A)$, or;
\item $n=t+6$ where $t\geq1$ is a triangular number and $b$ is in class $(3B)$.
\end{enumerate}
\end{theorem}

\begin{proof}
There is a unique block $B$ of $\Sym(n)$ covering $b$. If $B$ has weight $w$ then its defect groups are conjugate to the Sylow $2$-subgroups of $\Sym(2w)$, so the defect groups of $b$ are isomorphic to the Sylow $2$-subgroups of $\Alt(2w)$. Therefore we must have $w=3$ and the defect groups are of order $8$.

The $2$-blocks of $\Sym(n)$ are labelled by $2$-cores, which are precisely the triangular partitions. In \cite[Lem. 2.4, Thm 4.2]{Scopes} Scopes gives a method of showing that certain blocks of $\Sym(n)$ and $\Sym(n+m)$ of equal weight with $m\geq w$ are Morita equivalent with the same decomposition matrices. From this we get that all $2$-blocks of symmetric groups of weight $3$, other than perhaps the blocks labelled $\emptyset$ of $\Sym(6)$ and $(1)$ of $\Sym(7)$, are Morita equivalent to the block $(2,1)$ of $\Sym(9)$. 

The (principal) blocks of $\Alt(6)$ and $\Alt(7)$ covered by these blocks $\emptyset$ and $(1)$ can be checked individually and found to be as stated; note also that $\Alt(6)\cong \PSL_2(9)$. Suppose then that $B$ is Morita equivalent to the block $(2,1)$ of $\Sym(9)$, so has the following decomposition matrix:
\begin{equation*}
\begin{pmatrix}
1&.&.\\1&.&.\\1&1&.\\1&1&.\\1&.&1\\1&.&1\\1&1&1\\1&1&1\\.&1&.\\.&1&.\\
\end{pmatrix}.
\end{equation*}
Since $b$ has defect $3$ it has exactly five ordinary characters, so the ten ordinary characters of $B$ must fuse in the obvious way, giving that $b$ has decomposition matrix that of class $(3B)$.
\end{proof}

Additionally, the only odd covers of alternating groups are the exceptional covers $3\cdot\Alt(6)$ and $3\cdot\Alt(7)$, whose Sylow $2$-subgroups are also dihedral of order $8$. Checking these in GAP or the Modular Atlas shows that $3\cdot\Alt(6)$ has two additional blocks of defect $3$ in class $(3K)$, while $3\cdot\Alt(7)$ has two in class $(2B)$.

The Sylow $2$-subgroups of $\Alt(2w)$ are not semidihedral for any $w$, so there are no blocks of alternating groups with semidihedral defect groups.

\subsection{Sporadic Groups}

None of the sporadic groups has dihedral Sylow $2$-subgroups, so only non-principal blocks may have dihedral defect groups, and only the Mathieu group $M_{11}$ has semidihedral Sylow $2$-subgroups. Landrock \cite{Landrock} described the non-principal $2$-blocks of all sporadic groups, and their defect groups. There are several blocks with dihedral defect groups -- they are all of order only $8$, but we give their decomposition matrices anyway -- and there are two blocks with semidihedral defect groups, of order $16$. The ordinary character degrees can be found using GAP and its Atlas database of ordinary character tables, and, in most cases this is sufficient to deduce the decomposition matrix and Morita equivalence class. We also check any additional $2$-blocks of odd covers of the sporadic groups with defect at least $3$ but not maximal; these only occur for $3\cdot Fi^\prime_{24}$, and it can be seen from the number of characters and heights that the defect groups of these blocks are indeed dihedral.

\begin{theorem}
The blocks of sporadic groups with dihedral defect groups (of order at least $8$), with their decomposition matrices and character degrees, are as follows.

\bigbreak
\begin{center}
\begin{tabular}{|L|LL|}
\hline
Fi_{23}\ (2B)&97\,976\,320&166\,559\,744\\
\hline
97\,976\,320&1&.\\
166\,559\,744&.&1\\
166\,559\,744&.&1\\
264\,536\,064&1&1\\
264\,536\,064&1&1\\
\hline
\end{tabular}
\end{center}
\bigbreak

\begin{center}
\begin{tabular}{|L|LL|}
\hline
B\ (2B)&2\,642\,676\,197\,359\,616&9\,211\,433\,539\,600\,384\\
\hline
2\,642\,676\,197\,359\,616&1&.\\
9\,211\,433\,539\,600\,384&.&1\\
9\,211\,433\,539\,600\,384&.&1\\
11\,854\,109\,736\,960\,000&1&1\\
11\,854\,109\,736\,960\,000&1&1\\
\hline
\end{tabular}
\end{center}
\bigbreak

\begin{center}
\begin{tabular}{|L|LLL|}
\hline
Fi_{24}^\prime\ (3A)&38\,467\,010\,560&38\,641\,860\,608&107\,008\,229\,376\\
\hline
38\,641\,860\,608&.&1&.\\
77\,108\,871\,168&1&1&.\\
145\,650\,089\,984&.&1&1\\
184\,117\,100\,544&1&1&1\\
222\,758\,961\,152&1&2&1\\
\hline
\end{tabular}
\end{center}
\bigbreak

\begin{center}
\begin{tabular}{|L|LLL|}
\hline
O'N\ (3K)&10\,944&13\,376&13\,376\\
\hline
10\,944&1&.&.\\
13\,376&.&1&.\\
13\,376&.&.&1\\
26\,752&.&1&1\\
37\,696&1&1&1\\
\hline
\end{tabular}
\end{center}
\bigbreak

\begin{center}
\begin{tabular}{|L|LLL|}
\hline
He\ (3B)&1920&4352&4608\\
\hline
1920&1&.&.\\
4352&.&1&.\\
6272&1&1&.\\
6528&1&.&1\\
10\,880&1&1&1\\
\hline
\end{tabular}
\quad
\begin{tabular}{|L|LLL|}
\hline
Suz\ (3B)&66\,560&79\,872&102\,400\\
\hline
66\,560&1&.&.\\
79\,872&.&1&.\\
146\,432&1&1&.\\
168\,960&1&.&1\\
248\,832&1&1&1\\
\hline
\end{tabular}
\end{center}
\bigbreak

\begin{center}
\begin{tabular}{|L|LLL|}
\hline
Co_1\ (3B)&40\,370\,176&150\,732\,800&313\,524\,224\\
\hline
40\,370\,176&1&.&.\\
150\,732\,800&.&1&.\\
191\,102\,976&1&1&.\\
464\,257\,024&.&1&1\\
504\,627\,200&1&1&1\\
\hline
\end{tabular}
\end{center}
Additionally, $3\cdot Fi^\prime_{24}$ has two blocks with dihedral defect groups and the following decomposition matrix.
\bigbreak
\begin{center}
\begin{tabular}{|L|LL|}
\hline
3\cdot Fi^\prime_{24}\ (2A)&55\,349\,084\,160&80\,256\,172\,032\\
\hline
80\,256\,172\,032&.&1\\
80\,256\,172\,032&.&1\\
135\,605\,256\,192&1&1\\
135\,605\,256\,192&1&1\\
215\,861\,428\,224&1&2\\
\hline
\end{tabular}
\end{center}
\bigbreak
\end{theorem}

\begin{theorem}
The blocks of sporadic groups with semidihedral defect groups, with their decomposition matrices and character degrees, are as follows.
\bigbreak
\begin{center}
\begin{tabular}{|L|LLL|}
\hline
M_{11} \ (3B_1)&1&10&44\\
\hline
1&1&.&.\\
10&.&1&.\\
10&.&1&.\\
10&.&1&.\\
11&1&1&.\\
44&.&.&1\\
45&1&.&1\\
55&1&1&1\\
\hline
\end{tabular}
\quad
\begin{tabular}{|L|LLL|}
\hline
HN \ (3B_1)\text{ or }(3D)&214\,016&1\,361\,920&2\,985\,984\\
\hline
214\,016&1&.&.\\
1\,361\,920&.&1&.\\
1\,361\,920&.&1&.\\
1\,361\,920&.&1&.\\
1\,575\,936&1&1&.\\
2\,985\,984&.&.&1\\
3\,200\,000&1&.&1\\
4\,561\,920&1&1&1\\
\hline
\end{tabular}
\end{center}
The Monster group has a block with one of the following decomposition matrices:
\bigbreak
\begin{center}
\begin{tabular}{|L|LLL|LLL|}
\hline
M\ (3B_2)\text{ or }(3C_{2,1})&\varphi_1&\varphi_2&\varphi_3&\varphi_1&\varphi_2&\varphi_4\\
\hline
5\,514\,132\,424\,881\,463\,208\,443\,904&1&.&.&1&.&.\\
5\,514\,132\,424\,881\,463\,208\,443\,904&1&.&.&1&.&.\\
5\,514\,132\,424\,881\,463\,208\,443\,904&1&.&.&1&.&.\\
9\,416\,031\,858\,681\,585\,751\,556\,096&.&1&.&.&1&.\\
14\,930\,164\,283\,563\,048\,960\,000\,000&1&1&.&1&1&.\\
124\,058\,385\,593\,021\,471\,188\,320\,256&.&1&1&.&.&1\\
129\,572\,518\,017\,902\,934\,396\,764\,160&1&1&1&1&.&1\\
138\,988\,549\,876\,584\,520\,148\,320\,256&1&2&1&1&1&1\\
\hline
\end{tabular}
\end{center}
where $\varphi_1=5\,514\,132\,424\,881\,463\,208\,443\,904,$ $\varphi_2=9\,416\,031\,858\,681\,585\,751\,556\,096,$ $\varphi_3=114\,642\,353\,734\,339\,885\,436\,764\,160,$ $\varphi_4=124\,058\,385\,593\,021\,471\,188\,320\,256$.
\end{theorem}

The block of $HN$ is in one of the two Morita equivalence classes with its decomposition matrix, while the principal block of $M_{11}$ is known to be Morita equivalent to that of $\PSL_3(3)$ (a calculation in Magma). As seen above, the decomposition matrix of the block of M cannot be identified from just the ordinary character degrees; neither can it be identified by tensoring any of the possible projective characters of this block or those of the defect zero blocks of M with irreducible characters.

\section{Semidihedral Class $(2B_1)$}\label{2B1}

As according to the latter half of Proposition \ref{reduction3}, we consider the blocks of quasi-simple groups with dihedral defect groups in class $(3K)$, as have been described throughout Section \ref{proof}, and show that those with defect at least $4$ are not covered by blocks with semidihedral defect groups in class $(2B_1)$. There are such blocks in class $(3K)$ with defect only $3$ in $3\cdot\Alt(6)$ and $O'N$, and the semidihedral class $(2B_1)$ for defect $4$ does occur as a block of $3\cdot M_{10}$ covering this block of $3\cdot\Alt(6)$, while the covering block of $O'N.2$ has dihedral defect groups (a simple calculation in Magma). Other than these there are the quasi-isolated blocks of odd covers of $\PSL_n(\varepsilon q)$ and the unipotent blocks of $E_7(q)$, both for $q\equiv-1\bmod4$; we must also consider the non-quasi-isolated blocks of quasi-simple groups of Lie type that are Morita equivalent to these blocks via Theorem \ref{Bonnafe-Rouquier}.  

\begin{proposition}
Let $b$ be a block of a finite quasi-simple group of Lie type $G$ with dihedral defect groups in class $(3K)$. If $B$ is the block of a group $G.2$ covering $b$, then $B$ does not have semidihedral defect groups in class $(2B_1)$.
\end{proposition}

\begin{proof}
If $G.2$ did not stabilise $b$, then $B$ and $b$ would have the same defect groups, so assume that $G.2$ stabilises $b$. Assume also that $G.2$ does not stabilise all the characters of $b$ -- certainly the case if $B$ has semidihedral defect groups. Then $G.2$ defines some $\sigma\in\Out(G)$ of order $2$ that stabilises $b$. 

Let $G$ be defined over $q$, a power of a prime $p\neq2$, and let $\mathbf{G}$ be the corresponding simply connected simple algebraic group, with Steinberg endomorphism $F$ so that $G=\mathbf{G}^F/Z$, where $Z$ is a central subgroup; we may assume that $\abs{Z}$ is a power of $2$. Let $b_{\mathbf{G}^F}$ be the block of $\mathbf{G}^F$ containing $b$, with semisimple label $s\in(\mathbf{G}^*)^F$ according to Lusztig series. Let $\mathbf{L}^*$ be a minimal (not necessarily proper) $F$-stable Levi subgroup of $\mathbf{G}^*$ containing $C_{\mathbf{G}^*}(s)$, so $b_{\mathbf{G}^F}$ is Morita equivalent to a quasi-isolated block $b_{\mathbf{L}^F}$ of $\mathbf{L}^F$ by Theorem \ref{Bonnafe-Rouquier}, and $b$ is Morita equivalent to the block $b_L$ of $L=\mathbf{L}^F/Z$ contained in $b_{\mathbf{L}^F}$ by Theorem \ref{BonnafeRouquiercentral}. 

Suppose that $Z$ is trivial, so $b_{\mathbf{L}^F}$ itself has dihedral defect groups in class $(3K)$. If $b_{[\mathbf{L}^F,\mathbf{L}^F]}$ is a block of $[\mathbf{L}^F,\mathbf{L}^F]\trianglelefteq\mathbf{L}^F$ covered by $b_{\mathbf{L}^F}$, then by Remark \ref{reductionremark} it has trivial or dihedral defect groups. Since $[\mathbf{L}^F,\mathbf{L}^F]$ is simply connected, by \cite[Prop. 12.14]{Malle-Testerman}, and $b_{[\mathbf{L}^F,\mathbf{L}^F]}$ is also quasi-isolated, by Section \ref{proof} its defect groups are not dihedral, and if they were trivial then the defect groups of $b_{\mathbf{L}^F}$ would be abelian. Hence $Z$ must be non-trivial, so $\mathbf{G}$ is of type $A$, $B$, $C$, $D$, or $E_7$.

In case $[\mathbf{L}^F,\mathbf{L}^F]$ is not itself quasi-simple, consider the quasi-isolated blocks of quasi-simple groups with trivial defect groups, described in \cite[Lemma 5.2]{Eaton-Kessar-Kulshammer-Sambale}: there are such blocks of $G_2(q)$, $F_4(q)$, $E_6(q)$, and $E_8(q)$. But $G_2$, $F_4$, and $E_8$ cannot be contained in any proper Levi subgroup of a simple algebraic group, and since $\mathbf{G}\neq E_8$ a Levi subgroup of $\mathbf{G}$ cannot contain both $E_6$ and another simple component. Therefore $[\mathbf{L}^F,\mathbf{L}^F]$ is quasi-simple.

Consider the possible Levi subgroups $\mathbf{L}$ of $\mathbf{G}$. First, $[\mathbf{L}^F,\mathbf{L}^F]$ cannot be of type $B_n(q)$, $C_n(q)$, $D_n(q)$, or $^2\!D_n(q)$, as the only quasi-isolated blocks of these groups are the principal blocks, and the Sylow $2$-subgroups are such that the defect groups of $b_L$ could not be dihedral. This leaves $\mathbf{L}$ of type $A$, $E_6$, or $E_7$.

If $[\mathbf{L}^F,\mathbf{L}^F]=E_7(q)_{sc}$ then $\mathbf{L}^F=\mathbf{G}^F=E_7(q)_{sc}$. Then as in Theorem \ref{E7theorem}, since $b$ is in class $(3K)$ we must have $q\equiv-1\bmod4$, so $q$ is an odd power of $p$. Hence $\sigma$ is not a field automorphism and must be the diagonal automorphism, giving that $G.2=E_7(q)_{ad}$ and $B$ has dihedral defect groups as described in Section \ref{E7}.

If $[\mathbf{L}^F,\mathbf{L}^F]=E_6(\varepsilon q)_{sc}$, where $\varepsilon=\pm1$ and $E_6(-q)$ denotes $^2\!E_6(q)$, then $\mathbf{G}^F=E_7(q)_{sc}$ and $\mathbf{L}^F=E_6(\varepsilon q)_{sc}.(q-\varepsilon)$. Then $\mathbf{L}^F/Z$ has $E_6(\varepsilon q)_{sc}$ as a normal subgroup, and Remark \ref{reductionremark} implies that $b_L$ and the blocks of $E_6(\varepsilon q)_{sc}$ that it covers have the same defect groups and are Morita equivalent. But $E_6(\varepsilon q)_{sc}$ has no quasi-isolated blocks with dihedral defect groups in class $(3K)$, by Section \ref{proof}.

Otherwise, $[\mathbf{L}^F,\mathbf{L}^F]=\SL_m(\varepsilon q^a)$ for some $m\geq2$, $a\geq1$. If $\mathbf{L}^F$ is such that it contains $\GL_m(\varepsilon q^a)$, then since $b_{\mathbf{L}^F}$ is quasi-isolated it is the principal block. Hence $b_L$ is also the principal block of $L=\mathbf{L}^F/Z$, which then has $\PGL_m(\varepsilon q^a)$ as a quotient. For the defect groups of $b_L$ to be dihedral we require $m=2$, but then $b_L$ would have two simple modules, as the principal block of $\PGL_m(\varepsilon q^a)$ does in Theorem \ref{classification}, and not be in class $(3K)$. Hence $\mathbf{L}^F$ is a product of $\SL_m(\varepsilon q^a)$ and a torus. Since $b_{\mathbf{L}^F}$ is quasi-isolated, Theorem \ref{typeAdihedral} implies that $m/2$ is odd and $q^a\equiv-1\bmod4$, so $q$ is an odd power of $p$ and $\sigma$ is not a field automorphism.

Take $T$ to be the torus part, so $\abs{\mathbf{L}^F}=\abs{\SL_m(\varepsilon q^a)}\cdot\abs{T}$. So that the defect groups of $b_L$ are dihedral we must have $\abs{Z}=\abs{Z(\SL_m(\varepsilon q^a))}_2\cdot\ \abs{T}_2=2\cdot\abs{T}_2$. We consider the different possible $G$:

\textit{(i) G of Type B or C}: Carter \cite{Cartercentralisers} describes the maximal-rank subgroups of $\mathbf{G}^F$ for $\mathbf{G}$ classical. If $G$ is of type $B$ or $C$ then $\abs{T}_2\geq2$ by \cite[Props 9 \& 11]{Cartercentralisers}; but $\abs{Z}=2$, so this is not possible. 

\textit{(ii) G of Type A}: If $\mathbf{G}^F=\SL_n(\varepsilon q)$, then by \cite[Props 7 \& 8]{Cartercentralisers} we must have $[\mathbf{L}^F,\mathbf{L}^F]=\SL_m(\varepsilon q^{n/m})$ and $\abs{T}=(q^{n/m}-\varepsilon)/(q-\varepsilon)$ where $n/m$ is odd. Denote $\mathbf{GL}=\GL_n(k)$ and let $B_{GL}$ be a block of $\mathbf{GL}^F$ covering $b_{\mathbf{G}^F}$. Since $b_{\mathbf{L}^F}$ is quasi-isolated, as described in Section \ref{typeA} we have that $C_\mathbf{GL}(\tilde{s})^F\cong\GL_2(\varepsilon q^{n/2})$, where $\tilde{s}$ is a representative of $s$, and the defect groups of $B_{GL}$ are the Sylow $2$-subgroups of $\GL_2(\varepsilon q^{n/2})$. Quotienting by the $2$-part of the centre, consider the block $B_{GL/2}$ of $\GL_n(\varepsilon q)/2$ contained in $B_{GL}$. This group has $\PGL_n(\varepsilon q)$ as a quotient by the odd part of the centre, so $B_{GL/2}$ has defect groups the Sylow $2$-subgroups of $\PGL_2(\varepsilon q^{n/2})$, which are dihedral. If $\sigma$ is a diagonal automorphism then $G.2$ is an odd-index subgroup of $\GL_n(\varepsilon q)/2$, and $B_{GL}$ can be chosen so that $B_{GL/2}$ covers $B$; hence $B$ also has dihedral defect groups. 

We now show that $\sigma$ cannot be a graph automorphism. The graph automorphism of $\mathbf{GL}^F$ maps each conjugacy class to its inverse, so sends any $\chi\in\Irr(\mathbf{GL}^F)$ to its complex conjugate $\overline{\chi}$. By \cite[Lem. 4.2]{Srinivasan-Vinroot}, if $\chi\in\mathcal{E}(\mathbf{GL}^F,\tilde{s})$ then $\overline{\chi}\in\mathcal{E}(\mathbf{GL}^F,\tilde{s}^{-1})$; in particular if $\tilde{s}$ and $\tilde{s}^{-1}$ are not conjugate in $(\mathbf{GL}^*)^F$ then $\chi$ and $\overline{\chi}$ are in different blocks. Note that $\mathbf{GL}\cong\mathbf{GL}^*$ and $\tilde{s}$ is conjugate to $\tilde{s}^{-1}$ in $\mathbf{GL}^F$ if and only if they are conjugate in $\mathbf{GL}$. 

First suppose $F$ is the standard Frobenius endomorphism, so $C_\mathbf{GL}(\tilde{s})^F\cong\GL_2(q^{n/2})$. Then $\tilde{s}$ is $\mathbf{GL}$-conjugate to some $x=\mathrm{diag}(x_1,x_1,x_2,x_2,\dots,x_{n/2},x_{n/2})\in\mathbf{GL}$ where $x_1^q=x_2,\dots,x_{n/2}^q=x_1$, so $x_1^{q^{n/2}-1}=1$. Suppose that $x$ is conjugate to $x^{-1}$, which is equivalent to $\tilde{s}$ and $\tilde{s}^{-1}$ being conjugate in $\mathbf{GL}^F$. Conjugate elements have the same eigenvalues, so the entries of $x^{-1}$ must be those of $x$ permuted; hence $x_1\cdot x_1^{q^d}=x_1^{q^d+1}=1$ for some $d\leq n/2$. Let $r$ be a prime dividing the order of $x_1$, hence also dividing both $q^d+1$ and $q^{n/2}-1$; note that $r\neq2$ since the order of $\tilde{s}$ is odd. Then $q^d\equiv-1\bmod r$, so the order of $q\bmod r$ is even. But $q^{n/2}\equiv1\bmod r$ and $n/2$ is odd, a contradiction.

Now suppose $F$ is the twisted Steinberg endomorphism sending $(a_{ij})$ to $(a_{ij}^q)^{-T}$, so $\mathbf{GL}^F=\GU_n(q)$ and $C_\mathbf{GL}(\tilde{s})^F\cong\GU_2(q^{n/2})$. Then similarly $\tilde{s}$ is conjugate to $x=\mathrm{diag}(x_1,x_1,x_2,x_2,\dots,x_{n/2},x_{n/2})\in\mathbf{GL}$ and $x_1^{-q}=x_2,\dots,x_{n/2}^{-q}=x_1$, so, since $n/2$ is odd, $x^{q^{n/2}+1}=1$. Supposing $x$ is conjugate to $x^{-1}$ implies that $x_1^{(-q)^d+1}=1$ for some $d\leq n/2$. Then for a prime $r\neq2$ dividing the order of $x_1$ we have that $q^{n/2}\equiv-1\bmod r$, so the order of $q\bmod r$ is divisible by $2$ but not $4$. But also $(-q)^d\equiv-1\bmod r$, so if $d$ is odd then the order of $q\bmod r$ is odd, and if $d$ is even then the order is divisible by $4$, a contradiction in either case.

Therefore $\tilde{s}$ is not conjugate to $\tilde{s}^{-1}$, so $\chi$ and $\overline{\chi}$ are in different blocks; hence the graph automorphism of $\mathbf{GL}^F$ does not stabilise $B_{GL}$. Therefore the graph automorphism of $G$ does not stabilise $b$, nor does the product of the diagonal and graph automorphisms.

\textit{(iii) G of Type D}: Since $\abs{Z}\leq4$ we must have $\abs{T}_2\leq2$, so by \cite[Prop. 10]{Cartercentralisers} we have that $\mathbf{G}^F$ cannot be twisted, $[\mathbf{L}^F,\mathbf{L}^F]=\SL_m(q^{n/m})$ and $\abs{T}=q^{n/m}-1$ where $n/m$ is odd; note that $m$ and $n$ are even. The $\overline{\sigma}\in\Out(\mathbf{G}^F)$ corresponding to $\sigma\in\Out(G)$ stabilises $b_{\mathbf{G}^F}$, so by \cite[Prop. 4.9]{Ruhstorfer} there is some automorphism in $\Aut(\mathbf{G}^F)$ in the coset $\overline{\sigma}$ that stabilises $\mathbf{L}^F$. Considering the Dynkin diagram of $\mathbf{G}$, the Levi subgroup $\mathbf{L}$ contains the node of one leg but not the other. The graph automorphism of $\mathbf{G}$ swaps these two nodes, and since $n$ is even the two $A_{n-1}$ Levi subgroups are not conjugate, so it does not stabilise $\mathbf{L}^F$; hence $\sigma$ is not the graph automorphism. Nor is it the product of the diagonal and graph automorphisms, as any diagonal automorphism does stabilise $\mathbf{L}^F$. Then $\sigma$ is a diagonal automorphism, so there is a simple algebraic group $\mathbf{H}$ (a special orthogonal or half-spin group) with a central subgroup of order $2$ such that $\mathbf{H}^F/2=G.2$. Then Theorem \ref{BonnafeRouquiercentral} implies that $B$ is Morita equivalent to a quasi-isolated block of a Levi subgroup; since groups of type $D$ have no non-principal quasi-isolated blocks, this must be of type $A$. But no groups of type $A$ have quasi-isolated blocks with semidihedral defect groups in class $(2B_1)$; the quasi-isolated blocks of $\PSL_n(q)$ and $\SL_n(q)$ are given in Section \ref{typeA} and $\PGL_n(q)$ and $\GL_n(q)$ have no non-principal quasi-isolated blocks. Hence $B$ is not in class $(2B_1)$.

\textit{(iv) G of Type $E_7$}: If $\mathbf{G}^F=E_7(q)_{sc}$ then $\sigma$ must be the diagonal automorphism and $G.2=E_7(q)_{ad}$. Then if $B$ is not quasi-isolated it is Morita equivalent to a quasi-isolated block of a Levi subgroup of $G.2=(\mathbf{G}^*)^F$, and as above if this is of type $A$ or $D$ then $B$ is not in semidihedral class $(2B_1)$. If the Levi subgroup is of type $E_6$ then it is $E_6(\varepsilon q).(q-\varepsilon)$, containing a torus part divisible by $2$, so the defect groups of $B$ are a direct product and not semidihedral. 

If $B$ is unipotent then by Section \ref{E7} it does not have semidihedral defect groups. Otherwise if $B$ is quasi-isolated, then $C_\mathbf{G}(s^\prime)^F=\SL_6(q).\SL_3(q)$ or $\SU_6(q).\SU_3(q)$, by \cite[Table III]{Bonnafe}, where $s^\prime\in E_7(q)_{sc}$ is the semisimple label for $B$. Also $B$ must contain some character in $\mathcal{E}(E_7(q)_{ad},s^\prime)$, corresponding to a unipotent character $\chi_u$ of $C_\mathbf{G}(s^\prime)^F$, with degree $\chi_u(1)\cdot\abs{E_7(q)_{sc}:C_\mathbf{G}(s^\prime)^F}_{p^\prime}$ by Jordan decomposition in Theorem \ref{JordanDecomposition}. 

On the other hand, $[\mathbf{L}^F,\mathbf{L}^F]=\SL_m(\varepsilon q^a)$ and $\abs{T}$ is odd since $\abs{Z}=2$. As described in Section \ref{typeA}, in any case we have $(C_{\mathbf{G}^*}(s)^\circ)^F=\PGL_2(q^{m/2}).S$ for some torus $S$ of odd order. From the possible connected centralisers of $E_7$ listed in \cite{Deriziotis}, we must have $m/2=3$ and $\abs{S}=\Phi_3^2$, $\Phi_6^2$, or $\Phi_{12}$, where $\Phi_d$ denotes the $d$th cyclotomic polynomial evaluated at $q$. Then as in Section \ref{typeA} we have that $b_{\mathbf{L}^F}$, and hence $b_{\mathbf{G}^F}$, is Morita equivalent to the principal block of $\SL_2(q^3)$, and the ordinary character degrees of $b_{\mathbf{G}^F}$ are those of the principal block of $\SL_2(q^3)$ multiplied by $\abs{E_7(q)_{ad}:C_{\mathbf{G}^*}(s)^F}_{p^\prime}$ according to Jordan decomposition. Then the characters of $b$ correspond to those of the principal block of $\PSL_2(q^3)$, whose ordinary character degrees are $1$, $(q^3-1)/2$, $(q^3-1)/2$, $q^3$, and $q^3-1$ repeated. Hence the character degrees of $B$ are among these (and also $2(q^3-1)$ if they fuse on induction to $G.2$) multiplied by the factor of $\abs{E_7(q)_{ad}:C_{\mathbf{G}^*}(s)^F}_{p^\prime}$. But the degree of each of the unipotent characters of $C_\mathbf{G}(s^\prime)^F$ is too small.

Indeed, we must have $\chi_u(1)\cdot\abs{E_7(q)_{sc}:C_\mathbf{G}(s^\prime)^F}_{p^\prime}=\chi(1)\cdot\abs{E_7(q)_{ad}:C_{\mathbf{G}^*}(s)^F}_{p^\prime}$, for some $\chi(1)=1$, $(q^3-1)/2$, $q^3$, $q^3-1$, or $2(q^3-1)$. As above, $C_{\mathbf{G}^*}(s)^F=\PGL_2(q^3).S.A(s)$, where $A(s)=C_{\mathbf{G}^*}(s)/C_{\mathbf{G}^*}(s)^\circ$; note that the order of $s$, and hence $A(s)$, must be odd. First letting $C_\mathbf{G}(s^\prime)^F=\SL_6(q).\SL_3(q)$, the orders of these groups give $\chi_u(1)\cdot\abs{S}\cdot A(s)=\chi(1)\cdot\Phi_1^6\Phi_2^3\Phi_3^2\Phi_4\Phi_5$. Considering the $2$-part of the equation gives $\abs{\chi_u(1)}_2=2^7\cdot\abs{\chi(1)\cdot\Phi_2^3}_2$. Note $\chi_u$ is a product of unipotent characters $\chi_5$ and $\chi_2$ of $\SL_6(q)$ and $\SL_3(q)$ respectively. Calculating these degrees using \cite[Section 13.8]{Carter} gives $\abs{\chi_2(1)}_2\leq\abs{\Phi_2}_2$, and the only possible $\chi_5$ with $2$-part greater than $2$ has degree $q^4\Phi_2^3\Phi_4\Phi_6$. But the power of $q$ on the right-hand side of the above equation is at most $q^3$, so this is not possible.

Similarly for $C_\mathbf{G}(s^\prime)^F=\SU_6(q).\SU_3(q)$ we get $\chi_u(1)\cdot\abs{S}\cdot A(s)=\chi(1)\cdot\Phi_1^3\Phi_2^6\Phi_4\Phi_6^2\Phi_{10}$ and the same argument applies. Hence $B$ is not quasi-isolated.

Therefore $B$ does not have semidihedral defect groups in class $(2B_1)$.
\end{proof}

\section{References}

\printbibliography[heading=none]

\noindent Norman Macgregor\\University of Birmingham\\\textit{E-mail address}: \href{mailto:nxm835@bham.ac.uk}{nxm835@bham.ac.uk}

\end{document}